\documentclass[smallextended,envcountsame,envcountsect]{svjour3}
\smartqed

\usepackage{amsmath,amssymb}
\usepackage[T1]{fontenc}
\usepackage{lmodern}
\usepackage{microtype}
\usepackage{cite}
\usepackage{xurl}
\usepackage{graphicx}
\usepackage{booktabs}
\usepackage{tabularx}
\usepackage{array}
\usepackage{enumitem}
\usepackage{algorithm}
\usepackage{algpseudocode}
\usepackage{tikz}
\usetikzlibrary{arrows.meta,positioning,patterns}
\usepackage{xcolor}
\usepackage{placeins}
\usepackage{flafter}
\usepackage[hidelinks]{hyperref}
\newcommand{\doi}[1]{\href{https://doi.org/#1}{\nolinkurl{https://doi.org/#1}}}
\hypersetup{pdftitle={Greedy sign-based box splitting},
 pdfauthor={Alexander Yu. Gornov, Tatiana S. Zarodnyuk, Anton S. Anikin, Alexander V. Gasnikov},
 pdfsubject={Exact solvable classes, sharp thresholds and order-oracle complexity},
 pdfkeywords={Sign-based optimisation; box localisation; comparison oracle; global optimisation}}

\newcolumntype{P}[1]{>{\raggedright\arraybackslash}p{#1}}
\newcolumntype{L}{>{\raggedright\arraybackslash}X}


\spnewtheorem{counterexample}[theorem]{Counterexample}{\bfseries}{\itshape}
\spnewtheorem{openproblem}[theorem]{Open problem}{\bfseries}{\itshape}

\newcommand{\R}{\mathbb{R}}
\newcommand{\Rn}{\mathbb{R}^n}
\newcommand{\E}{\mathbb{E}}
\newcommand{\Prob}{\mathbb{P}}
\newcommand{\eps}{\varepsilon}
\newcommand{\sgn}{\operatorname{sign}}
\newcommand{\diag}{\operatorname{diag}}
\newcommand{\ip}[2]{\langle #1,#2\rangle}
\newcommand{\ninf}[1]{\left\|#1\right\|_\infty}
\newcommand{\ninfw}[1]{\left\|#1\right\|_{\infty,w}}
\newcommand{\Bx}{\mathcal{B}}
\newcommand{\CSC}{\mathcal{U}}
\newcommand{\MSC}{\mathcal{M}}
\newcommand{\DSC}{\mathcal{D}}
\newcommand{\CMP}{\mathcal{C}}
\newcommand{\algname}[1]{{\normalfont\textsc{#1}}}
\newcommand{\CUBE}{\algname{Cube-Sign}}
\newcommand{\CUBEF}{\algname{Cube-Freeze}}
\newcommand{\CUBES}{\algname{Cube-Smooth}}
\newcommand{\CUBEV}{\algname{Cube-Vote}}
\newcommand{\CUBEC}{\algname{Cube-Compare}}

\definecolor{ndblue}{RGB}{31,86,140}
\definecolor{ndgreen}{RGB}{18,124,86}
\definecolor{ndred}{RGB}{176,44,44}
\definecolor{ndgrey}{RGB}{110,110,110}
\definecolor{ndlight}{RGB}{237,243,249}
\newcommand{\keybox}[1]{%
  \begin{center}\fcolorbox{ndblue!55}{ndlight}{%
  \parbox{0.95\columnwidth}{\small #1}}\end{center}}

\renewcommand{\subclassname}{{\bfseries Mathematics Subject Classification (2020)}\enspace}
\journalname{Journal of Global Optimization}

\begin{document}
\raggedbottom

\title{Greedy sign-based box splitting:\\ exact solvable classes, sharp thresholds
       and order-oracle complexity}

\titlerunning{Greedy sign-based box splitting}
\authorrunning{A. S. Anikin et al.}

\author{Anton S. Anikin \and \\
		Alexander Yu. Gornov \and \\
        Tatiana S. Zarodnyuk \and \\
        Alexander V. Gasnikov}

\institute{A. S. Anikin \and A. Yu. Gornov \and T. S. Zarodnyuk \at
  Matrosov Institute for System Dynamics and Control Theory,
  Siberian Branch of the Russian Academy of Sciences,
  134 Lermontov St., Irkutsk 664033, Russia
  \and
  A. V. Gasnikov \at
  Moscow Institute of Physics and Technology, 9 Institutskiy per.,
  Dolgoprudny, Moscow Region 141701, Russia
  \and
  A. V. Gasnikov \at
  Innopolis University, 1 Universitetskaya St., Innopolis 420500, Russia}

\date{}

\maketitle

\begin{abstract}
Greedy box splitting reads gradient signs at the centre and retains one
smaller box. Its summable travel budget makes an incorrect exclusion
irreversible. We determine when this geometry is safe. Halving is minimax
optimal in the sign-vector oracle model on a coordinatewise sign-consistent
class that need not be convex. A dimensionless sign defect yields a sharp
scalar envelope, attained by one smooth convex potential; a three-dimensional
trajectory ends farther from the target than its initial radius.
For quadratics, universal convergence with aspect $w$ and contraction $\beta$
holds exactly when the weighted absolute row defect is at most $2\beta-1$.
Perron weights minimise this defect. Adaptive half-sizes avoid initial
enlargement and, under persistent error, converge to an exact capped
fixed-point limit. For binary comparisons, an independent probe fraction $q$
permits universal localisation on common-target monotone sections exactly
when $q\le2\beta-1$; coupled probes give the sharp threshold $2/3$.
Matching information bounds distinguish comparisons from sign-vector queries.
Error bands give finite-time guarantees and sharp residual floors.
Extensions and reproducible experiments identify the assumptions and costs
of freezing, smoothing, voting, multistart and random-matrix certificates.
\keywords{Sign-based optimisation \and Box localisation \and Comparison oracle
\and Generalised diagonal dominance \and Sharp error bounds \and Global optimisation}
\subclass{90C26 \and 90C30 \and 65K05}
\end{abstract}

\section{Introduction}\label{sec:intro}

A box can become small for two very different reasons: it may localise a
solution, or it may simply shrink around the wrong point. We study the simplest
scheme that exposes this distinction. At the centre of a box, read the signs
of the partial derivatives, move towards the indicated corner, and retain a
smaller box. For a cube split into $2^n$ equal subcubes, this is exactly
\begin{equation}\label{eq:signdescent}
 \begin{aligned}
 c_{k+1}&=c_k-2^{-k-1}r_0s_k,\\
 s_k&\in\{-1,+1\}^n,\qquad
 s_{k,i}=\sgn\partial_if(c_k)\quad\text{if }\partial_if(c_k)\ne0.
 \end{aligned}
\end{equation}
Only one sign vector is needed at each iteration. There is no line search and
no list of alternative boxes. The diameter halves, but the discarded regions
are never revisited. The central question is therefore not whether the boxes
shrink, but \emph{when a sign is a valid certificate for discarding a region}.

\subsection{A two-dimensional warning}
\label{sec:early-example}\label{sec:counterex}
Consider the strongly convex quadratic
\begin{equation}\label{eq:opening-example}
 f(x)=\tfrac12x^\top Hx,\qquad
 H=\begin{pmatrix}1&1\\1&2\end{pmatrix},\qquad
 c_0=(0.9,-1),\quad r_0=1,\quad x^*=0.
\end{equation}
The initial cube contains $0$ on its boundary. Since
$\nabla f(c_0)=(-0.1,-1.1)$, the first centre is $(1.4,-0.5)$ and the new
half-side is $1/2$. The first coordinate has moved \emph{away} from the
minimiser. All remaining displacements in that coordinate sum to at most
$1/2$, so recovery is impossible. The trajectory is illustrated in Fig.~\ref{fig:counterex}.
In fact the limit is $(0.9,-0.5)$, with
argument error $0.9r_0$ and objective-gap ratio $41/101$.

\begin{figure}[htbp]
\centering\includegraphics[width=\textwidth]{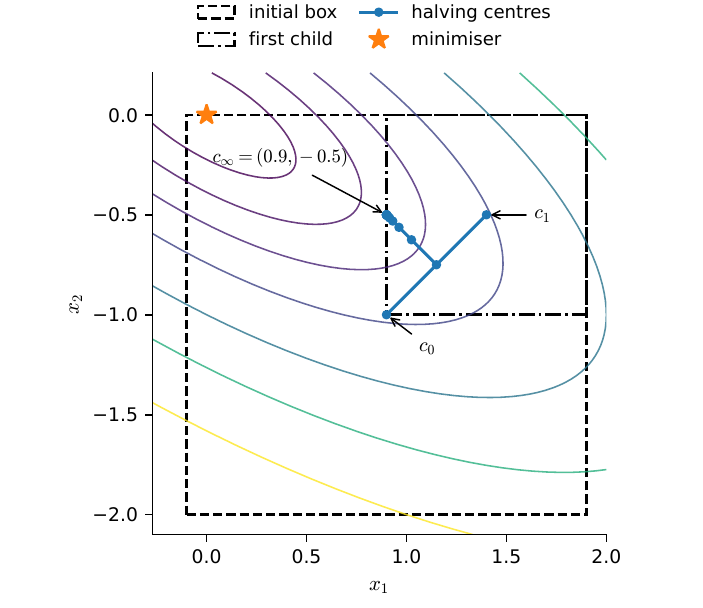}
\caption{Exact quadratic counterexample, $H=\left(\begin{smallmatrix}1&1\\1&2\end{smallmatrix}\right)$,
$c_0=(0.9,-1)$, $r_0=1$, $\beta=1/2$. Elliptical contours are level sets of
$f(x)=x^\top Hx/2$. The first child $[0.9,1.9]\times[-1,0]$ excludes the
minimiser. The trajectory converges to $(0.9,-0.5)$; all boxes are nested,
but they no longer contain the minimiser}
\label{fig:counterex}
\end{figure}

The obstruction is geometric, not a lack of convexity. In the first gradient
coordinate, the cross contribution $H_{12}c_{0,2}=-1$ outweighs
$H_{11}c_{0,1}=0.9$. Two remedies suggest themselves: retain overlapping
children, leaving room to correct a sign error, or change the aspect of the
box. With contraction $\beta\in[1/2,1)$ and aspect $w>0$, the exact quadratic
criterion will be
\begin{equation}\label{eq:opening-criterion}
 \theta_w(H):=\max_i\frac{\sum_{j\ne i}|H_{ij}|w_j}{H_{ii}w_i}
 \ \le\ 2\beta-1.
\end{equation}
For~\eqref{eq:opening-example}, $w=(1,1/\sqrt2)$ reduces the defect from $1$
to $1/\sqrt2$ and permits $\beta\ge(1+1/\sqrt2)/2$. The initial radius must
then be at least $\sqrt2$: preconditioning cannot be credited with success if
it excludes the target before the first query.

\subsection{Results and their scope}
The paper gives an exact account of this trade-off. On the class whose partial
derivatives have the correct coordinatewise signs, halving is minimax optimal
in the sign-vector oracle model. A defect parameter $\theta$ permits uncertain
signs near already accurate coordinates. We derive a scalar worst-case
recursion and realise its entire trajectory by one smooth convex potential;
three dimensions already allow a limiting error larger than the initial
radius. For quadratics,~\eqref{eq:opening-criterion} is necessary and sufficient
for \emph{every} initial box containing the target and \emph{every} resolution
of zero derivatives. Positive diagonal scaling is characterised by a
nonnegative comparison matrix. Adaptive half-sizes remove the need to enlarge
the initial box, and their limit under persistent coordinatewise error has an
exact capped fixed-point description.

For an oracle returning only which of two values is larger, we prove an exact
trade-off between probe distance and contraction. Coupling probe distance to
centre displacement gives the sharp threshold $\beta=2/3$. The class then
requires only strictly monotone coordinate sections with a common target,
not differentiability. A separate lower bound counts \emph{binary comparisons},
not sign vectors. Finally, a single perturbation model distinguishes exact
localisation from a noise floor. Freezing, smoothing, voting, multistart and
random-matrix certificates are stated with their additional assumptions and
costs in Sect.~\ref{sec:extensions} and proved in the appendices.

These are results for a structured global-localisation problem. They do not
make greedy single-box splitting a general-purpose global optimiser, nor do
they imply that a finite sample of Hessians certifies a uniform hypothesis.

\subsection{Relation to existing methods}
Sign-based optimisation includes Rprop~\cite{Riedmiller1993,Igel2003},
signSGD~\cite{Bernstein2018}, stochastic sign descent~\cite{Safaryan2021} and
error-feedback methods~\cite{Karimireddy2019}. Our distinctive restriction is
the summable geometric displacement budget coupled to nested boxes. Adaptive
sign methods can revise their effective travel budget; the present method
cannot. Moreover, error feedback generally uses compression magnitudes, which
are not supplied by a pure sign oracle. Majority voting itself is classical;
its role here is quantified through the cost of irreversible errors.

Comparison-oracle optimisation includes accelerated coordinate constructions
\cite{Lobanov2024}, gradient-direction estimation for quasi-convex objectives
\cite{GasnikovQuasi2024}, dichotomy on a square~\cite{Chervonenkis2025}, and
optimal gradient testing and estimation~\cite{TaoZhang2026}. Recent results
also address nonconvex stationarity~\cite{Wang2026} and preference-level-set
geometry~\cite{ScheinbergXiong2026}. We study a different question: when does one
coordinate comparison justify an irreversible box reduction? Our
$\eps$ measures distance to a fixed target. It must not be identified with
accuracy of a normalised gradient, a stationarity tolerance, or a preference
level-set gap.

Box subdivision and safe exclusion are classical in global optimisation,
including DIRECT~\cite{Jones1993}, diagonal approaches
\cite{SergeyevKvasov2017}, interval methods~\cite{Neumaier1990} and
branch-and-bound~\cite{HorstPardalos1995}. Such methods retain alternatives or
use bounds to justify exclusion. Computational background in global search,
optimal control and sparse quadratic optimisation is given in
\cite{Gornov2009,Zarodnyuk2010,Anikin2016,Gasnikov2021}.
Here only one child survives, so structural
assumptions replace backtracking. Likewise, comparison with general
localisation complexity~\cite{NemirovskiYudin1983} must account for our much
narrower function classes.

The matrix tools are the classical Perron--Frobenius and Collatz--Wielandt
theory~\cite{Varga2000,BermanPlemmons1994,Collatz1942,Wielandt1950}.
Strict generalised diagonal dominance is distinguished below from non-strict
scaled diagonal dominance~\cite{AhmadiMajumdar2019}. Coordinate sign
conditions relative to one target are also distinct from the two-point
conditions of nonlinear matrix theory~\cite{OrtegaRheinboldt1970,MoreRheinboldt1973}.
Connections with smoothing~\cite{NesterovSpokoiny2017,Duchi2015}, noisy binary
search~\cite{KarpKleinberg2007} and stochastic comparison averaging
\cite{Smirnov2024} are used only under explicitly matched oracle assumptions.

Section~\ref{sec:algorithm} develops the geometry and the exact class;
Sect.~\ref{sec:robust} gives the sharp defect envelope;
Sect.~\ref{sec:quad} studies matrix geometry;
Sect.~\ref{sec:compare} treats comparisons and error bands.
The remaining main sections summarise extensions, present controlled numerical
experiments, and identify the unresolved questions. Full proofs and additional
experiments follow the references.

\section{Geometry, exact localisation and information}
\label{sec:algorithm}\label{sec:csc}

\subsection{Notation and the oracle}\label{sec:notation}\label{sec:oracles}
For $c\in\Rn$ and a positive vector $r$, write
$\Bx(c,r)=\{x:|x_i-c_i|\le r_i,\ i=1,\dots,n\}$.
Products, inequalities and minima of vectors are componentwise.
The aspect $w\in\R^n_{>0}$ is fixed unless stated otherwise, and
\[
 \ninfw{u}=\max_i |u_i|/w_i,\qquad
 R_k=\beta^kr_0,\quad r_k=R_kw,\quad
 M_k=\ninfw{c_k-x^*}.
\]
Here $r_0$ in $R_k$ is a scalar; $r_k$ denotes the half-size vector.
For adaptive boxes we use a general vector $r^{\mathrm{init}}$ initially.
We assume $x^*\in B_0:=\Bx(c_0,r_0w)$ and $r_0>0$.
Differentiability on a closed box means differentiability on a neighbourhood
of it. Set $[t]_+=\max\{t,0\}$, $e_k=x^*-c_k$ and
$u_{k,i}=(c_{k,i}-x_i^*)/w_i$.

A sign query returns $\sgn\nabla f(c)\in\{-1,0,+1\}^n$; it does not return
magnitudes. Every zero is resolved to $-1$ or $+1$ by a rule depending only on
query points, sign answers and previous choices. Unless an existence statement
explicitly chooses a rule, the guarantees hold for \emph{all} such rules.
A sign-vector query and $n$ binary comparisons are different information units
(Table~\ref{tab:oracles}). Arithmetic outside the oracle costs $O(n)$ per
basic iteration; evaluating a dense quadratic gradient costs $O(n^2)$.

\begin{table}[htbp]
\caption{Oracle units. Equality in a binary comparison is resolved to either
binary answer; it does not create a third branch.}\label{tab:oracles}
\begin{tabularx}{\textwidth}{@{}P{28mm}LP{29mm}@{}}
\toprule
Oracle & Returned information & Cost of one iteration\\
\midrule
Sign vector & All $n$ coordinate signs, including zeros & One vector query\\
Binary comparison & Which of two values is larger & $n$ comparisons\\
Value oracle & A real function value & $2n$ values for the comparison rule\\
Full gradient & $n$ signed magnitudes & One gradient evaluation\\
\bottomrule
\end{tabularx}
\end{table}

\begin{algorithm}[htbp]
\caption{\CUBE$(\beta,w)$}\label{alg:cube}
\begin{algorithmic}[1]
\Require $c_0$, $r_0>0$, $w>0$, $1/2\le\beta<1$, number of steps $K$
\State $c\gets c_0$, $r\gets r_0w$
\For{$k=0,\dots,K-1$}
 \State Query $\sgn\nabla f(c)$ and resolve its zeros to obtain $s\in\{-1,+1\}^n$
 \State $c\gets c-(1-\beta)r\odot s$; \quad $r\gets\beta r$
\EndFor
\State \Return $c$ and the box $\Bx(c,r)$
\end{algorithmic}
\end{algorithm}

\subsection{The travel budget}\label{sec:locking}
\begin{lemma}[nestedness and sign descent]\label{lem:nested}\label{lem:sign-descent}
For every sequence of signs, Algorithm~\ref{alg:cube} satisfies
\begin{equation}\label{eq:err-rec}
 c_{k+1}=c_k-(1-\beta)R_kw\odot s_k,\qquad
 e_{k+1,i}=e_{k,i}+(1-\beta)r_{k,i}s_{k,i}.
\end{equation}
Its boxes are nested, its centres converge, and the total possible displacement
after centre $c_k$ is at most $r_{k,i}$ in coordinate $i$.
For $\beta=1/2$ it selects one of the $2^n$ equal subboxes; for $\beta>1/2$
the candidate children overlap.
\end{lemma}
\begin{proof}
In each coordinate, centre displacement plus new half-size is
$(1-\beta)r_{k,i}+\beta r_{k,i}=r_{k,i}$. The remaining displacement bound is
the geometric sum $\sum_{j\ge k}(1-\beta)r_{j,i}=r_{k,i}$.\qed
\end{proof}

We write $c_\infty=\lim_{k\to\infty}c_k$ and
$M_\infty=\ninfw{c_\infty-x^*}=\lim_{k\to\infty}M_k$.

Call a step \emph{wrong} in coordinate $i$ when $e_{k,i}\ne0$ and
$s_{k,i}=\sgn e_{k,i}$: the centre moves away from the target.
\begin{lemma}[locking]\label{lem:lock}
A wrong step implies, for every continuation,
\[
 \lim_{m\to\infty}|e_{m,i}|\ge |e_{k,i}|-(2\beta-1)r_{k,i}.
\]
Thus a wrong sign at $\beta=1/2$ preserves at least the pre-step coordinate
error forever.
\end{lemma}
\begin{proof}
The wrong displacement increases the distance by $(1-\beta)r_{k,i}$, while
all later displacements total at most $\beta r_{k,i}$. Subtract.\qed
\end{proof}

Figure~\ref{fig:scheme} illustrates the travel-budget obstruction in one coordinate.
\begin{figure}[htbp]
\centering\includegraphics[width=\textwidth]{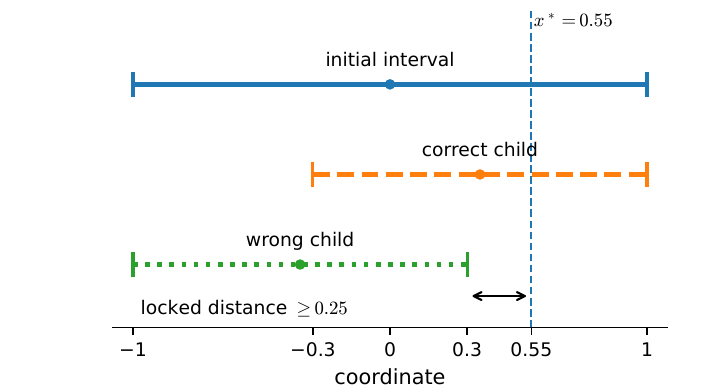}
\caption{One coordinate of the retention test, with $r=1$, $c=0$ and
$\beta=0.65$. The right child contains the target $x^*=0.55$; the wrong left
child ends at $2\beta-1=0.30$. Every subsequent box lies in that child, so the
limiting error is at least $0.55-0.30=0.25$. The central overlap, not convexity
alone, is the available protection against an uncertain sign.}\label{fig:scheme}\label{fig:locking}
\end{figure}

For repeated errors the following accounting identity is useful.
\begin{lemma}[coordinatewise error account]\label{lem:master}
Let $d_{k,i}=[|e_{k,i}|-r_{k,i}]_+$. At a correct step, or when $e_{k,i}=0$,
$d_{k+1,i}\le d_{k,i}$. At a wrong step,
\[
 d_{k+1,i}=[|e_{k,i}|-(2\beta-1)r_{k,i}]_+,\qquad
 d_{k+1,i}-d_{k,i}\le\min\{2(1-\beta)r_{k,i},|e_{k,i}|\}.
\]
Consequently, for $W_i$ the set of wrong steps and $x^*\in B_0$,
\[
 \lim_k|e_{k,i}|\le\sum_{k\in W_i}
       \min\{2(1-\beta)r_{k,i},|e_{k,i}|\}.
\]
\end{lemma}
The proof is in Appendix~\ref{app:exact}. The distinction between this
coordinatewise bound and a bound on the maximum over coordinates matters for
noisy voting.

\begin{proposition}[objective invariance and coordinate equivariance]
\label{lem:invariance}\label{rem:equivariance}
A differentiable transform $\varphi\circ f$ with $\varphi'>0$ leaves the sign
trajectory unchanged under the same tie rule. Under a change
$x=a+DPy$, with $D$ positive diagonal and $P$ a signed permutation, the box
trajectory is transformed correspondingly, provided the initial box,
aspect and \emph{tie rule} are transformed together.
\end{proposition}
The chain rule proves the assertion. An arbitrary coordinate-dependent
history rule need not respect even translations or dilations unless it is
conjugated with the coordinate change. For example, the fixed rule ``choose
$+1$ at zero'' is not reflection-equivariant for $f(t)=t^2$ at $c_0=0$.
All invariances here concern exact arithmetic; underflow can turn a nonzero
scaled derivative into a spurious zero.

\subsection{The exact sign-consistent class}
\begin{definition}\label{def:csc}
For $x^*\in B$, a differentiable $f$ belongs to $\CSC(x^*,B)$ when
\begin{equation}\label{eq:csc}
 \partial_if(x)(x_i-x_i^*)>0
 \quad\text{for every }x\in B\text{ and every }i\text{ with }x_i\ne x_i^*.
\end{equation}
Write $\CSC^\circ$ for the non-strict condition with $\ge0$.
\end{definition}

Every coordinate section of $f\in\CSC$ decreases strictly up to its fixed
target coordinate and increases strictly afterwards. Successively replacing
the coordinates of any $x\ne x^*$ by those of $x^*$ proves that $x^*$ is the
unique minimiser on $B$ (Proposition~\ref{prop:csc-unimodal}). Neither convexity
nor a lower bound on derivative magnitudes is required. Examples include
$\sum_i|x_i-x_i^*|^p$ for $p>1$ and the nonconvex
$\sum_i(1-e^{-(x_i-x_i^*)^2})$; smooth transforms with strictly positive derivative preserve the class.
Positive weighted sums and smooth compositions with strictly positive partial
derivatives give further examples
(Proposition~\ref{prop:algebra}).

\begin{theorem}[exact localisation]\label{th:csc}
If $f\in\CSC(x^*,B_0)$, then for any $1/2\le\beta<1$ and every tie rule,
\[
 x^*\in\Bx(c_k,r_k),\qquad M_k\le R_k=\beta^kr_0.
\]
Thus $K=\lceil\log(r_0/\eps)/\log(1/\beta)\rceil$ steps suffice for
$M_K\le\eps$ when $0<\eps<r_0$.
\end{theorem}
\begin{proof}
If $|e_{k,i}|\le r_{k,i}$ and $e_{k,i}\ne0$, the sign is correct and
\[
 |e_{k+1,i}|=\bigl||e_{k,i}|-(1-\beta)r_{k,i}\bigr|
 \le\max\{\beta,1-\beta\}r_{k,i}=\beta r_{k,i}.
\]
For $e_{k,i}=0$ either tie gives $(1-\beta)r_{k,i}\le\beta r_{k,i}$.
Induction gives retention and the bound.\qed
\end{proof}

Strict section monotonicity is \emph{not} enough for a derivative-sign oracle:
it permits a zero derivative away from the target. For example,
$F(t)=g(1-t)$ with $g(s)=s^4/4-2s^3/3+s^2/2$ satisfies
$F'(t)=-(1-t)t^2$ and has its unique minimum at $1$. Starting at $0$ with
$r_0=1$, a positive tie sends the centre to $-1/2$, then to $-1/4,-1/8,\dots$.
The limit is $0$, not $1$ (Example~\ref{ex:tie}).

For a local domain, only centres of admissible starting boxes need testing:
\begin{equation}\label{eq:Cw}
 C_w(B,x^*)=\{c:\exists r>0,\ x^*\in\Bx(c,rw)\subseteq B\}.
\end{equation}
\begin{theorem}[necessary and sufficient condition]\label{th:exact}
Let $f$ be differentiable on a neighbourhood of $B$, with fixed $w>0$ and
$\beta=1/2$. Convergence to $x^*$ from every admissible initial box, for every
tie rule, holds if and only if~\eqref{eq:csc} holds at every
$c\in C_w(B,x^*)$. In that case $M_k\le2^{-k}r_0$.
For $B=\Rn$, the condition is exactly $f\in\CSC(x^*,\Rn)$.
\end{theorem}
Sufficiency repeats the retention argument at reachable centres. Necessity
follows from the locking lemma applied to a violating centre used as the first
query. The complete proof and a local-domain example are in
Appendix~\ref{app:exact}.

\begin{corollary}[residual estimates]\label{cor:csc-f}
If additionally $\nabla f(x^*)=0$ and $\nabla f$ is $L$-Lipschitz on $B_0$,
then
\[
 \|\nabla f(c_k)\|_2\le L\|w\|_2R_k,\qquad
 0\le f(c_k)-f(x^*)\le\tfrac L2\|w\|_2^2R_k^2.
\]
\end{corollary}
Stationarity is essential here. For $f(t)=t+t^2/2$ on $[0,2]$, the boundary
minimiser is $0$; the centres $2^{-k}$ have an objective gap of order $2^{-k}$,
not $4^{-k}$, and the derivative tends to $1$.

\begin{theorem}[subgradients]\label{th:subgrad}
Let $f$ be convex and finite on a neighbourhood of $B_0$. If
\begin{equation}\label{eq:subcsc}
 g_i(x_i-x_i^*)>0\quad
 \text{for every }g\in\partial f(x),\ x\in B_0,\ x_i\ne x_i^*,
\end{equation}
then the same localisation bound holds for every subgradient selection and
every tie rule.
\end{theorem}
Only coordinate signs enter the proof. The condition includes
$\|x-x^*\|_1$, but not $\|x-x^*\|_\infty$: an inactive coordinate may have
zero subgradient despite nonzero error.

\subsection{Optimality in the sign-vector model}\label{sec:lower}
\begin{theorem}[sign-query lower bound]\label{th:lower}
For any deterministic algorithm using $k$ adaptive full sign-vector queries,
the worst-case supremum of its weighted output error is at least $2^{-k}r_0$,
even over translated spherical quadratics with target in $\Bx(c_0,r_0w)$.
Hence halving on $\CSC$ is minimax optimal, including the constant.
\end{theorem}
A single unknown target coordinate suffices: an adversary retains a half of
the feasible interval after each query. Its length remains at least
$2r_0/2^k$, and no output is closer than half that length to all feasible
targets. Equality answers do not help in the supremum, since they can be
avoided. The randomized expected-error lower bound is $r_0/2^{k+1}$, obtained
from a uniform prior in one target coordinate and fixed values in the others
(Appendix~\ref{app:exact}). This is not a lower bound for full gradients: for
a known identity Hessian, one gradient query recovers the target exactly.

\section{A sign defect and its sharp envelope}\label{sec:robust}

A wrong sign is expensive only in a coordinate that is still far from the
target. This observation turns a qualitative sign condition into a quantitative
measure of uncertainty.
\begin{definition}[defect class]\label{def:msc}
For $x^*\in B$ and $0\le\theta<1$, write $f\in\MSC_w(\theta,x^*,B)$ if $f$ is differentiable
on a neighbourhood of $B$ and
\begin{equation}\label{eq:msc}
 \frac{|x_i-x_i^*|}{w_i}>\theta\ninfw{x-x^*}
 \quad\Longrightarrow\quad \partial_if(x)(x_i-x_i^*)>0
 \qquad(x\in B).
\end{equation}
Inside the indicated band the derivative may have either sign or vanish.
\end{definition}
Thus $\MSC_w(0)=\CSC$, and the classes increase with $\theta$.
The condition is not merely a root-finding hypothesis.
\begin{proposition}[unique minimiser]\label{prop:clipped-path}
Every $f\in\MSC_w(\theta,x^*,B)$ with $\theta<1$ has the unique minimiser
$x^*$ on the box $B$.
\end{proposition}
To see the idea, travel from $x^*$ to any $x\ne x^*$ along
\[
 \gamma_i(t)=x_i^*+\sgn(x_i-x_i^*)\min\{|x_i-x_i^*|,tw_i\},
 \quad 0\le t\le\ninfw{x-x^*}.
\]
Every moving coordinate has the largest normalised error $t$, so each of its
contributions to $d f(\gamma(t))/dt$ is positive. This piecewise smooth path
stays in $B$ and strictly increases $f$. Appendix~\ref{app:envelope} gives the
details; convexity is not used.

Figure~\ref{fig:clipped-path} shows why the path is useful: once a coordinate
has reached its endpoint, it stops; every coordinate that still moves remains
outside the uncertain-sign band.
\begin{figure}[htbp]
\centering
\includegraphics[width=\textwidth]{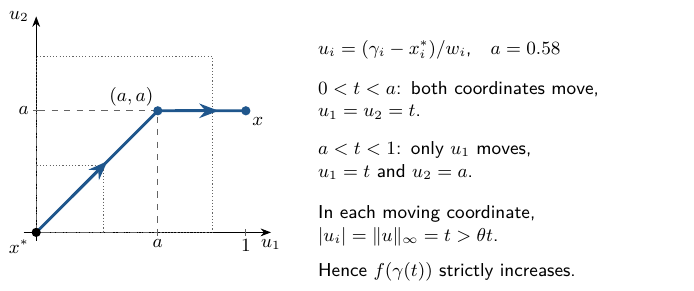}
\caption{The clipped path used to prove uniqueness in Proposition~\ref{prop:clipped-path},
shown in normalised coordinates with positive endpoint components. Dashed and
 dotted squares are maximum-norm level sets. A coordinate stops as soon as it
reaches its endpoint; all remaining moving coordinates have the current largest
normalised error. This is a proof path, not a sequence of optimisation iterates}
\label{fig:clipped-path}
\end{figure}
\FloatBarrier

\subsection{The scalar worst case}\label{sec:envelope}
\begin{proposition}[scalar envelope]\label{prop:envelope}
Let $1/2\le\beta<1$, $0\le\theta<1$, $x^*\in B_0$ and
$f\in\MSC_w(\theta,x^*,B_0)$. Define
\begin{equation}\label{eq:envelope}
 b_0=r_0,\qquad
 b_{k+1}=\max\{b_k-(1-\beta)R_k,\ \theta b_k+(1-\beta)R_k\}.
\end{equation}
Then $M_k\le b_k$ for every $k$. If $\theta\le2\beta-1$, then $b_k=R_k$.
Otherwise the index
\[
 J=\min\{k:b_k/R_k\ge 2(1-\beta)/(1-\theta)\}
\]
is finite, the first branch dominates from $J$ onwards, and
$b_k=b_J-R_J+R_k$ for $k\ge J$. In particular,
$M_\infty\le b_J-R_J$.
\end{proposition}
The elementary estimate behind this envelope is
\begin{equation}\label{eq:scalar}
 M_{k+1}\le\max\{M_k-h_k,\ \theta M_k+h_k\},\qquad h_k=(1-\beta)R_k.
\end{equation}
A correct sign decreases a non-overshooting coordinate by $h_k$. An uncertain
coordinate is at most $\theta M_k$ from the target and can increase by at most
$h_k$. Overshoots are covered by the second term, which is at least $h_k$.
Monotonicity in $M_k$ gives the envelope. Its branch-switch analysis is in
Appendix~\ref{app:envelope}.

\begin{theorem}[sharp residual for halving]\label{th:theta}
If $\beta=1/2$ and $0\le\theta\le1/2$, the preceding hypotheses imply
\begin{equation}\label{eq:theta-bound}
 M_k\le2^{-k}r_0+\theta r_0\quad(k\ge1),\qquad M_\infty\le\theta r_0.
\end{equation}
The coefficient of $\theta r_0$ in the limiting bound cannot be reduced,
even for positive definite quadratics in two dimensions.
\end{theorem}
Indeed, with
$H_\theta=\left(\begin{smallmatrix}1&\theta\\\theta&1\end{smallmatrix}\right)$
and $c_0=r_0(\eta,-1)$ for $0<\eta<\theta$, the first gradient coordinate is
negative although $c_{0,1}>0$. Locking gives $M_\infty\ge\eta r_0$;
let $\eta\uparrow\theta$ (Proposition~\ref{prop:tight}).
The finite-time coefficient is different: at $\theta=0.4$, $\eta=0.39$, the
first error is $0.89r_0$, not bounded by
$[2^{-1}+\theta(1-2^{-1})]r_0=0.7r_0$.

\subsection{A smooth function realises the whole envelope}
A scalar bound is useful, but sharpness requires that its adversarial signs
come from the gradient of \emph{one fixed function}. Here they do.
\begin{theorem}[attainability]\label{th:attain}
For $1/2\le\beta<1$ and $0<\theta<1$ with $\theta>2\beta-1$, let $J$ be the
switching index in Proposition~\ref{prop:envelope} for $r_0=1$.
There exist $n=J+1$, a convex $C^\infty$ function on $\Rn$ with a globally
Lipschitz gradient and unique minimiser $0$, and a starting centre in
$[-1,1]^n$, such that the function belongs to $\MSC_{\mathbf1}(\theta,0,\Rn)$
and a deterministic tie rule yields
\[
 \|c_k\|_\infty=b_k\quad(k\ge0).
\]
When $\theta\le2\beta-1$, a one-dimensional quadratic attains $b_k=R_k$.
\end{theorem}

The construction separates integrability from trajectory design. Let the random variables
$Z_i$, $1\le i\le n$, be independent with a smooth positive density on $(\theta,1)$,
flat at both endpoints, and set
\begin{equation}\label{eq:G-potential}
 G(x)=\E\max_i\frac{x_i^2}{Z_i^2},\qquad
 F(x)=\psi(G(x)),\qquad \psi(s)=\int_0^s e^{-1/t}\,dt.
\end{equation}
The integrand at $t=0$ is defined by its continuous extension. Then
\begin{equation}\label{eq:exact-sign-field}
 \partial_iF(x)=0\ \text{if }|x_i|\le\theta\|x\|_\infty,
 \qquad x_i\partial_iF(x)>0\ \text{otherwise}.
\end{equation}
A coordinate too small to win the random maximum has zero derivative.
Outside that band it wins with positive probability. The flat outer function
makes the potential smooth at the origin without changing these signs.
A backward interval construction allocates one coordinate to each pre-switch
peak. The full regularity, convexity and Lipschitz arguments, followed by the
trajectory construction, appear in Appendix~\ref{app:envelope}.

\begin{example}[a fixed tie rule in dimension three]\label{ex:witness}
For $\beta=1/2$, $\theta=4/5$, $r_0=1$, take the potential above and choose
$s_i=-1$ at every zero derivative. The starting centre and its first two steps
are
\[
 (1,4/5,27/50)\ \longmapsto\ (1/2,13/10,26/25)
 \ \longmapsto\ (3/4,21/20,129/100).
\]
Here $J=2$, and for every $k\ge2$,
$M_k=26/25+2^{-k}$. The limiting error $1.04$ exceeds the initial radius $1$.
\end{example}

Figure~\ref{fig:theta} compares this exact trajectory with the scalar envelope.
\begin{figure}[htbp]
\centering\includegraphics[width=\textwidth]{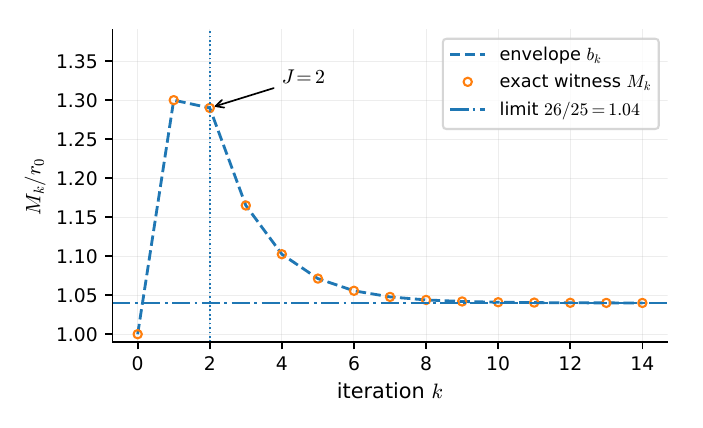}
\caption{The exact rational trajectory in Example~\ref{ex:witness} and the
scalar envelope coincide. After the switch at $J=2$, the error is the sum of
the locked residual $1.04$ and the remaining radius $2^{-k}$. This is a
smooth convex example, not an arbitrary vector field.}\label{fig:theta}
\end{figure}

Sharpness is dimension-free: the dimension may depend on the parameters and
ties may be selected as in the theorem. It is not a claim of sharpness for
every fixed dimension, every tie rule, or the strongly convex quadratic
subclass. Random quadratic trajectories can therefore lie well below this
sharp general envelope without contradicting it. Computationally, if the switch
has not yet been detected after $K$ steps, $b_K$ need not bound the limit;
$b_K+R_K$ is a certified upper bound because $R_K$ is the remaining budget.

\section{Matrix geometry: fixed and adaptive boxes}\label{sec:quad}

\subsection{A computable certificate}\label{def:defect}
For a symmetric matrix $H$ with $H_{ii}>0$, define
\[
 \begin{gathered}
 D=\diag(H_{11},\dots,H_{nn}),\qquad B_H=H-D,\\
 A=D^{-1}|B_H|,\qquad \theta_w(H)=\max_i\frac{(Aw)_i}{w_i}.
 \end{gathered}
\]
The absolute value is entrywise. The dimensionless matrix $A$ measures how
much the other coordinates can corrupt each gradient sign.
\begin{theorem}[averaged-Hessian criterion]\label{th:hess}
Let $f\in C^2$ on a neighbourhood of a box $B$, with
$x^*\in B$ and $\nabla f(x^*)=0$. For $x\in B$ put
$\bar H(x)=\int_0^1\nabla^2f(x^*+t(x-x^*))\,dt$.
If $\bar H(x)$ has positive diagonal and
$\theta_w(\bar H(x))\le\theta<1$ throughout $B$, then
$f\in\MSC_w(\theta,x^*,B)$.
\end{theorem}
Indeed, $\nabla f(x)=\bar H(x)(x-x^*)$. In coordinate $i$, the absolute cross
term is at most $\bar H_{ii}(x)w_i\theta\ninfw{x-x^*}$, so the diagonal
contribution fixes the sign outside the defect band. Ordinary Hessian values
at a finite set of sampled points do not establish this uniform condition.
Neither does minimising the aspect separately at each point supply a single
aspect valid everywhere. Convexity is not required: for
$f=\sum_i(1-e^{-(x_i-x_i^*)^2})$, the averaged Hessian is
$\diag(2e^{-(x_i-x_i^*)^2})$ and has zero defect, although the ordinary Hessian
may be indefinite.

\begin{theorem}[optimal aspect]\label{th:perron}
For symmetric $H$ with positive diagonal,
\[
 \theta_{\rm opt}(H):=\min_{w>0}\theta_w(H)=\rho(A).
\]
The minimum is attained, including the reducible case. The symmetric matrix
$S=D^{-1/2}|B_H|D^{-1/2}$ is similar to $A$. On each connected component,
choose a positive eigenvector $v$ for the \emph{largest algebraic} eigenvalue
of $S$ and put $w=D^{-1/2}v$; concatenate the componentwise vectors.
\end{theorem}
The proof is the Collatz--Wielandt formula, with a componentwise construction
(Appendix~\ref{app:matrix}). A global positive eigenvector need not exist when
component spectral radii differ. Selecting an eigenvalue of largest
\emph{modulus} is unsafe on bipartite graphs, where both $\rho$ and $-\rho$
occur. Numerically, the returned weights should be checked through
$\max_i(Aw)_i/w_i$ rather than treating a floating-point eigenvalue as an exact
certificate.\label{rem:bipartite}

\subsection{The exact quadratic threshold}
\begin{theorem}[universal quadratic criterion]\label{th:quad-iff}
Let $f(x)=\tfrac12(x-x^*)^\top H(x-x^*)$, with $H=H^\top$ and positive
diagonal. For fixed $w>0$ and $1/2\le\beta<1$, convergence to $x^*$ from
\emph{every} initial box containing $x^*$ and for \emph{every} tie rule holds
if and only if
\begin{equation}\label{eq:quad-exact}
 \theta_w(H)\le2\beta-1.
\end{equation}
Under this condition $H$ is positive definite and $M_k\le R_k$.
\end{theorem}
The sufficiency combines the averaged-Hessian certificate and the scalar
envelope. For necessity, take a row with defect $\theta_i>2\beta-1$ and choose
$t\in(2\beta-1,\min\{\theta_i,1\})$. Set
$u_i=tr_0w_i$ and $u_j=-\sgn(H_{ij})r_0w_j$ for $j\ne i$, and start at
$x^*+u$. Its $i$-th gradient component is
$H_{ii}r_0w_i(t-\theta_i)<0$, despite $u_i>0$. The first sign is wrong and
locks a residual of at least $(t-2\beta+1)r_0$ in the weighted norm.
There is no tie in the offending coordinate.

\begin{corollary}[when a fixed geometry exists]\label{cor:hmat}
Some $w>0$ and $\beta<1$ give universal convergence if and only if
$\rho(A)<1$, equivalently strict generalised diagonal dominance of $H$.
With optimal weights, any $\beta\ge(1+\rho(A))/2$ suffices.
For unmodified halving, universal convergence of a quadratic requires $H$
to be diagonal.
\end{corollary}
Strictness cannot be dropped. A positive definite $3\times3$ matrix with
unit diagonal and every off-diagonal entry $1/2$ is non-strictly diagonally
dominant but has $\rho(A)=1$. Nor does Jacobi convergence alone certify this
method: Jacobi involves signed off-diagonal entries, whereas $A$ takes their
absolute values (Appendix~\ref{app:matrix}).\label{rem:sdd}

\begin{remark}[defect versus total cost]\label{rem:cost-w}
The aspect minimising $\theta_w$ need not minimise the total number of
iterations for a common unweighted accuracy. To cover an entire cube of radius
$R$, a box of aspect $w$ needs scalar radius $R/\min_iw_i$.
A sufficient count for unweighted error $\eps<R$ is
\[
 K(w)=\left\lceil
 \frac{\log(R\kappa_w/\eps)}{\log(2/(1+\theta_w(H)))}
 \right\rceil,\qquad \kappa_w=\frac{\max_iw_i}{\min_iw_i},\quad\theta_w(H)<1.
\]
For $H_n=\operatorname{tridiag}(-1,2,-1)$,
$w_i=\sin(\pi i/(n+1))$ and $\theta_{\rm opt}=\cos(\pi/(n+1))$.
The unrounded aspect overhead
$\log\kappa_w/\log(1/\beta)$ is $\Theta(n^2\log n)$ at
$\beta=(1+\theta_{\rm opt})/2$. It is not merely logarithmic in $n$.
\end{remark}

\subsection{How badly fixed geometry can fail}
The example in the introduction has a moderately conditioned positive definite
Hessian yet locks a substantial error. A higher-dimensional construction
approaches the largest possible error $2r_0$ allowed by the initial radius and
the entire travel budget.
\begin{theorem}[near-maximal failure]\label{th:impossible}
For every $n\ge5$, fixed $\beta\in[1/2,1)$, fixed $w>0$, and
$\eta\in(0,1/4)$, there are a positive definite quadratic and an admissible
initial box such that
\[
 M_\infty>\left(2-\frac{4(1+\eta)}n\right)r_0.
\]
One may use $H_\epsilon=\mathbf1\mathbf1^\top+
\epsilon(I-\mathbf1\mathbf1^\top/n)$ with sufficiently small $\epsilon>0$.
\end{theorem}
The signs initially follow the almost rank-one coupling rather than the small
individual coordinate error; enough irreversible displacement is accumulated
before they can change. The proof in Appendix~\ref{app:matrix} makes the
choice of $\epsilon$ and all quantifiers explicit. This is an existence
result for each fixed geometry, not a claim that every matrix outside the
certified class has near-maximal error.\label{rem:crit}

\subsection{Adaptive half-sizes without initial enlargement}
A known comparison matrix permits a less conservative geometric response.
For an arbitrary initial half-size vector $r^{\mathrm{init}}>0$, put
\begin{equation}\label{eq:newbox}
 \begin{split}
 t_k&=\min\{r_k,Ar_k\},\qquad
 h_k=(r_k-t_k)/2,\qquad r_{k+1}=(r_k+t_k)/2,\\
 c_{k+1}&=c_k-h_k\odot s_k,\qquad r_0=r^{\mathrm{init}}.
 \end{split}
\end{equation}
A coordinate with excessive cross-coupling is temporarily left unchanged;
other coordinates shrink until that obstruction disappears.
Figure~\ref{fig:adaptive-geometry} illustrates the resulting coordinatewise cut.
\begin{figure}[htbp]
\centering
\includegraphics[width=\textwidth]{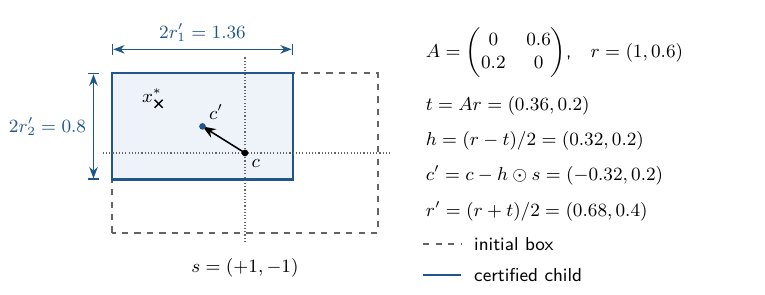}
\caption{One adaptive step for $H=\left(\begin{smallmatrix}1&0.6\\0.6&3\end{smallmatrix}\right)$
and $c=0$. The observed signs select the child
$[-1,0.36]\times[-0.2,0.6]$, with unequal coordinate contractions.
The marker $x^*=(-0.65,0.37)$ is one compatible minimiser.
Each retained interval accounts for the largest coordinate error at which its
sign can be wrong; no initial enlargement is needed}
\label{fig:adaptive-geometry}
\end{figure}
\begin{theorem}[adaptive retention and convergence]\label{th:adaptive}
For a quadratic as above with $\rho(A)<1$, rule~\eqref{eq:newbox} preserves
$x^*$ and nestedness from every box $\Bx(c_0,r^{\mathrm{init}})\ni x^*$.
If $Aw\le qw$ for $w>0$ and $q<1$, then
\begin{equation}\label{eq:adapt-rate}
 r_k\le C_0\left(\frac{1+q}{2}\right)^kw,\qquad
 C_0=\max_i r_i^{\mathrm{init}}/w_i.
\end{equation}
Among rules choosing $r'$ from $H,r$ before observing the current sign vector,
using displacement $r-r'$ and requiring retention for every admissible centre
and every tie, the componentwise smallest possible $r'$ is
\begin{equation}\label{eq:one-step-opt}
 r'_i=\tfrac12\bigl(r_i+\min\{r_i,(Ar)_i\}\bigr).
\end{equation}
\end{theorem}
A wrong sign can occur only at $|c_i-x_i^*|\le(Ar)_i$; retention also gives
$|c_i-x_i^*|\le r_i$. Thus its worst distance is $t_i$, and
$t_i+h_i=r'_i$. A correct step is safe as well. Monotonicity of the half-size
map gives~\eqref{eq:adapt-rate}. The one-step lower bound uses a rowwise
adversarial centre; it is not an optimality claim over algorithms exploiting
the joint observed sign vector (Appendix~\ref{app:matrix}).

\begin{corollary}[exact limiting half-sizes with bounded error]
\label{cor:matrix-noise}\label{cor:capped-limit}
Under the hypotheses of Theorem~\ref{th:adaptive}, suppose the gradient answer has component error at most $H_{ii}\xi_i$,
where $\xi\ge0$ is fixed. Replace $Ar_k$ in~\eqref{eq:newbox} by
$Ar_k+\xi$. Retention and nestedness still hold. The half-sizes converge to
the unique solution $v\in[0,r^{\mathrm{init}}]$ of
\begin{equation}\label{eq:capped-limit}
 v=\min\{r^{\mathrm{init}},Av+\xi\}.
\end{equation}
Moreover,
\begin{equation}\label{eq:capped-rate}
 0\le r_k-v\le\left(\frac{I+A}{2}\right)^k(r^{\mathrm{init}}-v),\qquad
 v\le (I-A)^{-1}\xi.
\end{equation}
The centres converge and $|c_\infty-x^*|\le v$.
If $r^{\mathrm{init}}\ge(I-A)^{-1}\xi$, then $v=(I-A)^{-1}\xi$.
\end{corollary}
The capped map in~\eqref{eq:capped-limit} is a contraction in the weighted
norm whenever $Aw\le qw$, $q<1$. It therefore gives a computable exact radius
limit, not only an upper bound. For instance, with
$A=\left(\begin{smallmatrix}0&0.05\\5&0\end{smallmatrix}\right)$,
$\xi=(0.03,0.04)$ and $r^{\mathrm{init}}=(0.02,1)$, the limit is
$v=(0.02,0.14)$, strictly smaller than the uncapped resolvent bound in both
coordinates. Appendix~\ref{app:matrix} proves the finite-time estimate.

Under the positive-diagonal and stationarity assumptions of Theorem~\ref{th:hess},
the same retention argument applies to a nonlinear problem when one known
nonnegative $A$ uniformly bounds the normalised off-diagonal absolute values
of its averaged Hessians. Pointwise spectral-radius bounds without a common
majorant do not provide this adaptive certificate.

\section{Comparisons and imperfect signs}\label{sec:compare}

\subsection{A derivative-free retention rule}
A binary order oracle returns $+1$ when $f(a)>f(b)$ and $-1$ when
$f(a)<f(b)$; at equality either answer is admissible. Define the target class
without referring to derivatives.
\begin{definition}[common-target monotone sections]\label{def:cmp}
For $x^*\in B$, a function $f:B\to\R$ belongs to $\CMP(x^*,B)$ if, for every fixed choice of
the other coordinates, its $i$-th section is strictly decreasing up to and
including $x_i^*$, and strictly increasing from and including $x_i^*$.
Equivalently, whenever the corresponding points belong to $B$,
\begin{equation}\label{eq:cmp}
 \begin{aligned}
 a<b\le x_i^*&\ \Longrightarrow\ f(x_{-i},a)>f(x_{-i},b),\\
 x_i^*\le a<b&\ \Longrightarrow\ f(x_{-i},a)<f(x_{-i},b).
 \end{aligned}
\end{equation}
\end{definition}
Continuity is not assumed; the inequalities involving the target itself are
part of the definition. They imply that $x^*$ is the unique minimiser.
For differentiable functions, $\CSC\subseteq\CMP\subseteq\CSC^\circ$,
but neither inclusion may be reversed in general
(Proposition~\ref{prop:cmp-vs-csc}). The zero-derivative example from
Sect.~\ref{sec:algorithm} belongs to $\CMP$ and illustrates why comparisons
can be informative when the derivative sign is not.

For $q\in(0,1]$, query
\begin{equation}\label{eq:compare-rule}
 s_{k,i}=\operatorname{cmp}\bigl(c_k+qr_{k,i}e_i,
                                  c_k-qr_{k,i}e_i\bigr),\qquad
 c_{k+1,i}=c_{k,i}-(1-\beta)r_{k,i}s_{k,i}.
\end{equation}
Here $e_i$ is a coordinate unit vector, not the error vector $e_k$.
Both probes lie in the current box; the radius contracts by $\beta$.
The natural coupled rule, denoted \CUBEC, takes $q=1-\beta$, so that the probe
distance equals the centre displacement.

\begin{theorem}[sharp probe--contraction trade-off]\label{th:compare}\label{th:probe-sharp}
For fixed $1/2\le\beta<1$ and $0<q\le1$, rule~\eqref{eq:compare-rule}
retains the target and converges from every initial box for every
$f\in\CMP(x^*,B_0)$ and every tie rule if and only if
\begin{equation}\label{eq:probe-threshold}
 q\le2\beta-1.
\end{equation}
Under this condition $M_k\le R_k$. In particular, the coupled rule has the
sharp threshold $\beta\ge2/3$ and uses
$n\lceil\log(r_0/\eps)/\log(1/\beta)\rceil$ binary comparisons to achieve
$M_k\le\eps$.
\end{theorem}
If $|c_{k,i}-x_i^*|\ge qr_{k,i}$, both probes are on the same monotone side
(or one is the target), so the answer is correct. Otherwise either direction
has new error at most $(q+1-\beta)r_{k,i}$, which is at most $\beta r_{k,i}$
exactly under~\eqref{eq:probe-threshold}. For necessity, choose
$t\in(2\beta-1,q)$ and the one-dimensional function
\[
 F_t(x)=\begin{cases}(x-t)^2,&x\le t,\\ C(x-t)^2,&x\ge t,\end{cases}
 \qquad C>\left(\frac{q+t}{q-t}\right)^2.
\]
Starting from $c_0=0$, $r_0=1$, the first comparison is wrong and locks an
error at least $t-(2\beta-1)>0$. This witness is strongly convex and $C^{1,1}$.
Appendix~\ref{app:comparison} gives the full proof, including the equality
case. The threshold concerns this simultaneous single-box rule, not all
algorithms for unimodal search.

Figure~\ref{fig:probe} displays the admissible probe--contraction region.
\begin{figure}[htbp]
\centering\includegraphics[width=\textwidth]{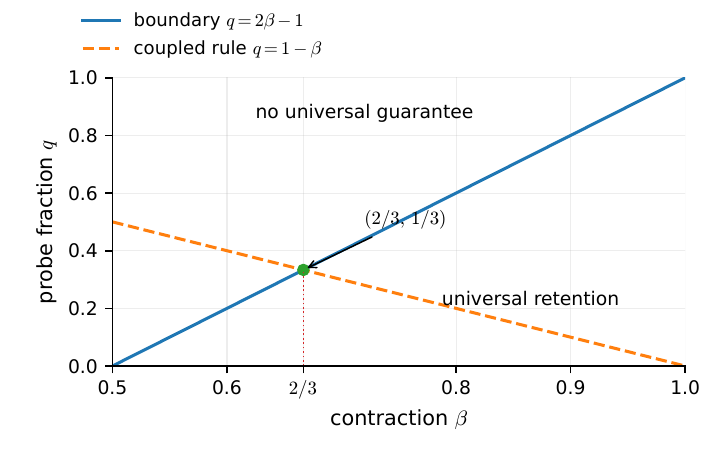}
\caption{Exact probe--contraction trade-off on the common-target section class.
The admissible region lies below $q=2\beta-1$. The coupled line $q=1-\beta$
first enters it at $(\beta,q)=(2/3,1/3)$. Smaller probes permit faster geometric
contraction but amplify value noise; the diagram states a theorem, not a
simulated phase boundary}
\label{fig:probe}
\end{figure}

Smaller $q$ permits $\beta$ closer to $1/2$, but increases sensitivity to value
noise. No fixed positive probe radius permits $\beta=1/2$ on the entire class
$\CMP$. Quadratics are more structured:
\begin{proposition}[polarisation]\label{prop:polarization}
For any quadratic and any $h>0$,
\begin{equation}\label{eq:polarization}
 f(c+he_i)-f(c-he_i)=2h\,\partial_if(c).
\end{equation}
Thus comparison and sign trajectories coincide in exact arithmetic with
matched tie rules. Their quadratic convergence criterion is
$\theta_w(H)\le2\beta-1$, independent of $q>0$.
\end{proposition}
Comparisons also survive every strictly increasing transformation of the
values, even a nondifferentiable one. For derivative signs, the corresponding
chain-rule statement needs $\varphi'>0$.

\subsection{The information cost}
\begin{theorem}[comparison lower bound]\label{th:cmp-lower}
Any deterministic algorithm guaranteed to localise every translated spherical
quadratic whose target lies in the initial box $B_0$ to weighted accuracy
$0<\eps<r_0$ needs at least
\[
 N\ge n\log_2(r_0/\eps)
\]
binary comparisons in the worst case. A randomised algorithm with a fixed
budget and success probability at least $1-\delta$ on every target satisfies
$N\ge n\log_2(r_0/\eps)+\log_2(1-\delta)$.
\end{theorem}
A depth-$N$ decision tree has at most $2^N$ leaves. Each output can cover at
most a fraction $(\eps/r_0)^n$ of the initial target volume; averaging the
same argument over algorithmic randomness proves the second statement.
This gives $\Theta(n\log(r_0/\eps))$ comparisons on $\CMP$ for the coupled
rule and exact leading constant on separable quadratics with $\beta=1/2$.
There is no conflict with the sign-vector lower bound: one sign vector contains
$n$ coordinate answers. The cost in raw function values is twice the number
of comparisons for the direct two-value implementation.

\subsection{One model for geometric and numerical uncertainty}\label{sec:inexact}
Suppose the chosen sign is correct whenever
\begin{equation}\label{eq:inexact-cond}
 |u_{k,i}|>\theta M_k+\eta_k,\qquad \eta_k\ge0.
\end{equation}
Here $\theta$ measures geometric uncertainty and $\eta_k$ is an absolute
error band in weighted coordinate units. Put
$\Delta=2\beta-1-\theta$.
\begin{theorem}[retention margin]\label{th:inexact}
If $0\le\theta\le2\beta-1$ and
$\eta_k\le\Delta R_k$ at each step, then $x^*$ remains in every box and
$M_k\le R_k$ for every admissible sequence of oracle answers.
\end{theorem}
The uncertain region has radius at most
$\theta R_k+\eta_k\le(2\beta-1)R_k$, exactly the overlap protection.
A positive persistent error cannot satisfy this condition indefinitely.
It does not, however, make the rest of the trajectory unanalysable.

\begin{theorem}[finite-time error and sharp persistent floor]\label{th:noise}
Under~\eqref{eq:inexact-cond}, assume $M_0\le r_0$ and
$0\le\theta\le2\beta-1$. Define
\begin{equation}\label{eq:noise-rec}
 d_0=0,\qquad d_{k+1}=\max\{d_k,\ \theta d_k+\eta_k-\Delta R_k\}.
\end{equation}
Then for $K\ge1$,
\begin{equation}\label{eq:noise-finite}
 M_K\le R_K+d_K\le R_K+
 \frac{\max_{0\le j<K}[\eta_j-\Delta R_j]_+}{1-\theta}.
\end{equation}
If $\eta_k\le\bar\eta$, then
\begin{equation}\label{eq:noise-limit}
 M_\infty\le\frac{\bar\eta}{1-\theta}.
\end{equation}
The coefficient $1/(1-\theta)$ is sharp already in one dimension for this
band model.
\end{theorem}
Substitution of $M_k\le R_k+d_k$ into
$M_{k+1}\le\max\{M_k-h_k,\theta M_k+\eta_k+h_k\}$ proves the recursion.
For sharpness, $f(x)=x^2/2$ with the biased sign oracle $\sgn(x-p)$,
$c_0=r_0=p>0$ and $\eta=(1-\theta)p$ satisfies the band condition but
converges to $p$, not to $0$. This is a worst-case floor, not an assertion that
every noisy trajectory attains it. Full proofs are in
Appendix~\ref{app:comparison}.

\subsection{Converting oracle error to a coordinate band}
\begin{proposition}[a quantitative sign-growth condition]\label{prop:eta}
Suppose, for $a_i>0$,
\begin{equation}\label{eq:growth}
 \sgn(x_i-x_i^*)\partial_if(x)\ge
 a_iw_i\left(\frac{|x_i-x_i^*|}{w_i}-\theta\ninfw{x-x^*}\right).
\end{equation}
If a derivative estimate has absolute component error at most $b_{k,i}$, then
\eqref{eq:inexact-cond} holds with
$\eta_k=\max_i b_{k,i}/(a_iw_i)$.
For a quadratic, one may take $a_i=H_{ii}$ and $\theta=\theta_w(H)$.
\end{proposition}
For quadratics this is an \emph{inequality} obtained by bounding each row defect
by the maximum row defect; equality need not hold. The condition is stronger
than membership in $\CSC$, where a correct derivative may be arbitrarily flat.

\begin{corollary}[central differences and noisy values]\label{cor:fd}
If the derivative of the $i$-th coordinate section is $L_i$-Lipschitz on the
probe segment and each value has absolute error at most $\nu$, then
\[
 \left|\frac{\widehat f(c+he_i)-\widehat f(c-he_i)}{2h}
              -\partial_if(c)\right|\le\frac{L_ih}{2}+\frac{\nu}{h}.
\]
For $L_i>0$ and $\nu>0$, the best radius for this bound is
$h_i=\sqrt{2\nu/L_i}$, giving $b_i=\sqrt{2L_i\nu}$.
Provided all probes are available and~\eqref{eq:growth} holds, retention is
certified while
\[
 R_k\ge\frac1\Delta\max_i\frac{\sqrt{2L_i\nu}}{a_iw_i}
 \quad(\Delta>0),
\]
and the persistent floor is at most
$(1-\theta)^{-1}\max_i\sqrt{2L_i\nu}/(a_iw_i)$.
\end{corollary}
This optimised fixed probe radius is not the coupled radius
$(1-\beta)r_{k,i}$. With that shrinking radius, the term $\nu/h$ grows; it
cannot be ignored. A comparison perturbed by an additive error bounded by
$\omega$ in the value difference contributes $\omega/(2h)$ instead.
For exact values, difference bias is proportional to $h$, but retention still
requires its coefficient to satisfy the margin condition. Quadratic difference
bias is identically zero by polarisation; floating-point cancellation is not.

High-probability derivative bounds transfer to the same high-probability
localisation statement \emph{only when the resulting bands also satisfy the
margin}. To control a run of $K$ steps, one may allocate failure probability
$\delta/(nK)$ to each coordinate estimate. For independent estimates of fixed
variance, attaining errors of order $R_k$ typically requires sample sizes of
order $R_k^{-2}$, so logarithmically many geometric iterations do not imply
logarithmic total sampling cost.

\section{Beyond deterministic signs}\label{sec:extensions}

The geometric principle survives several changes to the oracle. Their
hypotheses and costs must remain visible: skipping an uncertain sign,
averaging gradients, repeating a noisy answer and restarting an algorithm
are not interchangeable operations. Table~\ref{tab:mods} collects the main
guarantees; Appendix~\ref{app:mods} contains their definitions and proofs.

\begin{table}[htbp]
\caption{Extensions and their limitations. $K=\lceil\log(r_0/\eps)/\log(1/\beta)\rceil$;
$m_i(k)$ is the number of updates of coordinate $i$ before step $k$.}
\label{tab:mods}
\small\renewcommand{\arraystretch}{1.12}
\begin{tabularx}{\textwidth}{@{}P{20mm}LP{31mm}@{}}
\toprule
Rule & Main guarantee & Extra information or cost\\
\midrule
Overlap & For $\theta\le2\beta-1$, $M_k\le\beta^kr_0$.
For $\theta<\beta$, also
$M_k\le R_k+\lambda(r_0-R_k)$,
$\lambda=[(\theta+1-2\beta)/(1-\beta)]_+$ &
$K\le\lceil2\log(r_0/\eps)/(1-\theta)\rceil$ at $\beta=(1+\theta)/2$\\
\addlinespace
Freezing & Correct signs on the active set imply
$M_k\le\beta^{\min_i m_i(k)}r_0$ &
Gradient magnitudes; an activation frequency is needed for a linear rate\\
\addlinespace
Smoothing & If $f_\tau\in\CSC(x^*_\tau)$,
$\|c_k-x^*\|_\infty\le\|w\|_\infty R_k+\|x^*_\tau-x^*\|_\infty$ &
One exact smoothed gradient, or $m$ sampled gradients per step\\
\addlinespace
Majority & With $p_m\le e^{-2\gamma^2m}$,
$\E M_\infty\le2r_0\min\{1,np_m\}$ &
$m$ independent vector queries; fixed positive correctness margin $\gamma$\\
\addlinespace
Multistart & Some run succeeds with probability $1-(1-\pi)^M$ &
$M$ runs; success-basin mass $\pi$ is not known automatically\\
\bottomrule
\end{tabularx}
\end{table}

\paragraph{Freezing does not prove that a coordinate is safe.}
For a threshold $0<\epsilon\le1$, update only coordinates with
$|\partial_if(c_k)|\ge\epsilon\|\nabla f(c_k)\|_\infty$, and keep both the
centre and half-size of every other coordinate unchanged. The defining
condition of $\DSC_\epsilon$ requires sign consistency only on these active
coordinates, away from their target values. Retention follows by induction.
Without an assumption such as $m_i(k)\ge\nu k$ for all $i$, the resulting
bound does not imply a uniform linear rate. Empirical success of a threshold
is not a certificate that the function belongs to $\DSC_\epsilon$.

\paragraph{Smoothing changes the function and the cost.}
For uniform $u\in[-1,1]^n$, use a fixed radius in
$f_\tau(x)=\E f(x+\tau u)$. For the translated multimodal family
\[
 f(x)=\|x-x^*\|_2^2+A\sum_i[1-\cos(2\pi(x_i-x_i^*))],
\]
its smoothed derivative is $2t+c\sin(2\pi t)$ in each coordinate, where
\[c=2\pi A\,\operatorname{sinc}(2\pi\tau),\]
and
$\operatorname{sinc}(z)=\sin(z)/z$, continuously extended at zero.
It belongs to $\CSC$ exactly when
\[
 -1/\pi\le c<c_+,\qquad
 c_+=\min_{1/2<t<1}\frac{2t}{-\sin(2\pi t)}\approx1.465.
\]
This covers functions with exponentially many isolated local minima before
smoothing. The upper endpoint is strict; the lower endpoint is included.
Exact smoothing, Monte Carlo approximation and a radius shrinking with $k$
are three different oracles. A special two-gradient cancellation for this
family is discussed in Remark~\ref{rem:twoquery}; it is not a general-purpose
global optimisation oracle.

\paragraph{Noise accumulates over coordinates as well as time.}
If a coordinate answer is independently correct with conditional probability
at least $1/2+\gamma$, an odd majority of size
$m\ge\log(nK/\delta)/(2\gamma^2)$ gives $M_K\le R_K$ with probability at
least $1-\delta$. The expected per-coordinate limiting error is at most
$2p_mr_0$, but the maximum norm needs the factor $n$ in Table~\ref{tab:mods}.
With conditional independence across coordinates the sharper bound is
$2r_0[1-(1-p_m)^n]$. Add $R_K$ for finite $K$. Already for $n=3$, $p=0.1$,
one wrong first-step sign implies
$\E M_\infty\ge1-0.9^3=0.271>2p=0.2$ when $r_0=1$.
A fixed $\gamma>0$ is a noise-model assumption, not a consequence of a fixed
variance of additive gradient or value noise near a minimiser.

\paragraph{Random starts and random matrices have different quantifiers.}
If successful starting centres occupy fraction $\pi$ of a sampling region,
$M\ge\pi^{-1}\log(1/\delta)$ independent starts give at least one success
with probability $1-\delta$. Selecting the lowest final objective preserves a
function-value guarantee, but not an argument-distance guarantee without a
growth condition. In high dimensions a small basin can make $\pi$ exponentially
small.

A separate probabilistic question is whether a random Hessian admits a
\emph{universal} fixed geometry. For
$H=I+\sigma W/\sqrt n$, where $\sigma\ge0$ and $W$ is a Gaussian orthogonal ensemble matrix
with off-diagonal variance $1$ and diagonal variance $2$, put
\[
 \ell=\log(4n/\delta),\quad a=\sqrt{2(n-1)\ell},\quad
 b=2\sigma\sqrt{\ell/n},\quad m_0=\sqrt{2/\pi}.
\]
If $b<1$, then with probability at least $1-\delta$ all diagonal entries are
positive and
\begin{equation}\label{eq:main-GOE}
 \frac{\sigma[(n-1)m_0-a]_+}{\sqrt n(1+b)}
 \le\rho(A)\le
 \frac{\sigma[(n-1)m_0+a]}{\sqrt n(1-b)}.
\end{equation}
An upper bound below $1$ certifies both positive definiteness and the existence
of a safe geometry; a lower bound above $1$ rules out any universal fixed
geometry. If $0<\sigma=O(n^{-1/2})$ and $\log(n/\delta)=o(n)$, these bounds
identify the scale $\rho(A)\sim\sigma\sqrt{2n/\pi}$. For a Gaussian-weighted
$d$-regular graph with unit diagonal, the analogous bounds are
$\sigma[dm_0\pm\sqrt{2d\log(2n/\delta)}]/\sqrt d$ (truncate the lower bound
at zero); their relative error vanishes only if $\log(n/\delta)=o(d)$.
A deterministic maximum degree yields only a one-sided row-sum bound.
A star has spectral scale $\sigma$, not $\sigma\sqrt d$, illustrating why the
regularity assumption matters. The full statements and proofs are in
Appendix~\ref{app:random}.

\section{Controlled numerical experiments}\label{sec:experiments}

The experiments answer three different questions: whether the implementation
respects the proved invariants, how the irreversible mechanism appears in
concrete examples, and how selected finite-budget runs compare with reference
methods. They are not a search for a universal winning optimiser. Complete
instance data, trajectories, seeds and scripts are supplied in Online Resource~1.
Further experiments are in Appendix~\ref{app:experiments}; the execution
protocol is in Appendix~\ref{app:repro}.

\subsection{What is measured}
All publication experiments use double precision unless explicitly labelled
as exact rational computations. Random cases use independent, case-addressed
streams derived from the seed $20260919$, so adding a different case does not
change an existing instance. Within a comparison, methods share the target,
starting centre and external error metric. A fixed-aspect box is initialised
with $r_0=\max_i|c_{0,i}-x_i^*|/w_i$ unless a larger radius is stated.
This target-dependent initialisation is a controlled-test device, not a rule
available to an optimiser with an unknown solution.

The primary plotted error is the unweighted distance divided by the common
initial unweighted error. Weighted errors are used only for testing a weighted
theorem, and are explicitly identified. Each basic iteration uses one gradient
or sign-vector query; each comparison iteration uses $n$ binary comparisons,
or $2n$ raw values. Matrix setup, products $Ar_k$, repeated votes, Monte Carlo
samples and restarts are not silently counted as free gradients. Logarithmic
plots display errors smaller than $10^{-16}$ at $10^{-16}$; the saved data
are not floored. Tiny tails below floating-point resolution have no empirical
asymptotic interpretation.

\subsection{Algebraic and adversarial checks}
The verification suite checks the statements before testing performance.
Its exact rational part covers the error-account identity, the full
backward-constructed envelope trajectories, quadratic polarisation using
actual function values, and the sharp probe boundary. The floating-point part
checks fixed and adaptive retention, nestedness, the time-varying error band,
Perron weights for disconnected and bipartite graphs, and the capped limiting
half-sizes. Table~\ref{tab:verification} gives the executed counts. Each case
contains multiple steps; the rows are different tests, not independent
statistical observations of a common success probability.

\begin{table}[htbp]
\caption{Reproducible verification suite. All listed tests passed. Exact checks
use rational arithmetic; numerical tolerances and maximum observed discrepancies
are recorded in \texttt{verification.json}.}
\label{tab:verification}
\begin{tabularx}{\textwidth}{@{}LrP{30mm}@{}}
\toprule
Property & Cases & Arithmetic\\
\midrule
Coordinatewise error account & 2000 & Exact\\
Entire envelope trajectory & 120 & Exact\\
Quadratic polarisation & 120 & Exact\\
Probe threshold and boundary & 80 & Exact\\
Fixed quadratic retention & 600 & Double precision\\
Adaptive retention, nestedness and rate & 600 & Double precision\\
Time-varying noise band & 800 & Double precision\\
Capped noisy limit and finite-time rate & 320 & Double precision\\
Perron components and bipartite cases & 200 & Double precision\\
Late switch, reference step and input validation & 6 & Double precision\\
\bottomrule
\end{tabularx}
\end{table}

Figure~\ref{fig:theta} is stronger than an approximate agreement of two
curves: its coordinates and signs are rational, and $M_k=b_k$ is verified
exactly. No numerical integration of the potential is needed to evaluate its
analytically established sign field. Conversely, the scalar noise-band tests
use admissible abstract sign sequences; they are not presented as gradients
of newly constructed functions. The smooth attainability theorem and its
proof supply integrability where that claim is made.

\subsection{A matched comparison inside and outside the exact class}
We use two $50$-dimensional positive definite quadratics, with $200$ gradient
evaluations per method and a common target sampled uniformly from
$[-0.6,0.6]^{50}$. The first Hessian is diagonal with entries geometrically
spaced from $1$ to $1000$. The second is a dense symmetric perturbation of the
identity; its precise generating formula and saved matrix appear in
Appendix~\ref{app:repro}. Both start at $c_0=0$.

Reference parameters are fixed in advance: gradient descent uses $1/L$ with
$L=\lambda_{\max}(H)$; Adam~\cite{KingmaBa2015} uses step $0.05$, coefficients $(0.9,0.999)$ and
stabiliser $10^{-8}$; signGD uses $0.35/(k+1)^{0.7}$. The iRprop$^-$
implementation follows~\cite{Igel2003}: when a derivative changes sign, the
step is reduced and that coordinate update is skipped. Its initial step is
$0.1$, increase/decrease factors are $1.2$ and $0.5$, and step limits are
$10^{-14}$ and $1$. These choices do not claim optimal tuning for any method.

Figure~\ref{fig:bench-in} reports the matched comparison on the diagonal instance.
\begin{figure}[htbp]
\centering\includegraphics[width=\textwidth]{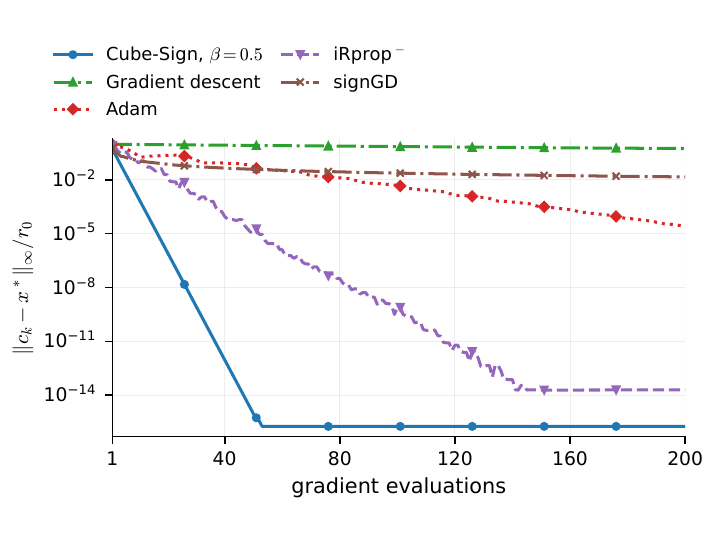}
\caption{Diagonal quadratic, $n=50$, condition number $1000$. The basic method
uses only signs and follows the predicted geometric localisation despite the
large spread of diagonal curvature. Reference methods use the same target,
initial centre and gradient-evaluation budget. Small displayed plateaus reflect
finite precision and the specified algorithm parameters, not proved error floors}
\label{fig:bench-in}
\end{figure}

For the diagonal instance, sign consistency is certified without inspecting
an empirical convergence curve. Figure~\ref{fig:bench-in} illustrates this
invariance to diagonal curvature. The dense instance in
Fig.~\ref{fig:bench-out} does not satisfy the halving criterion. Its locked
halving trajectory is consistent with the necessity theorem, whereas the
outcome of a more overlapping rule on this particular start remains a
finite-instance observation. An uncertified run must not be turned into a
universal convergence claim merely because its plotted error decreases.

\begin{figure}[htbp]
\centering\includegraphics[width=\textwidth]{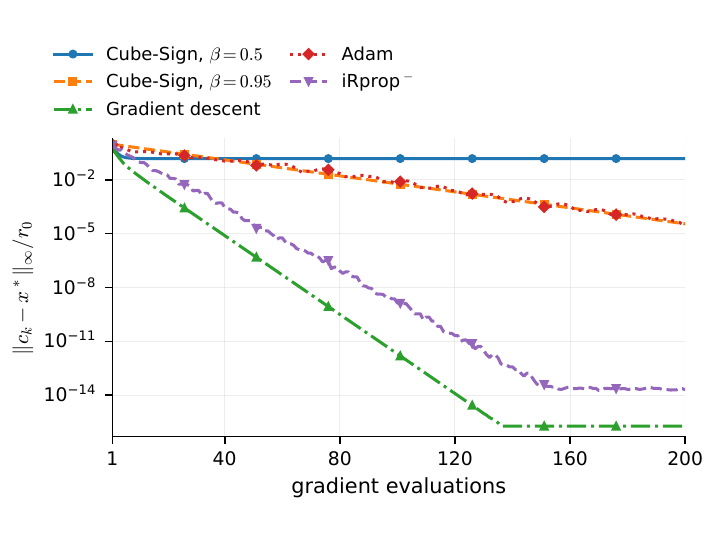}
\caption{Dense positive definite quadratic, $n=50$. Halving and the more
conservative choice $\beta=0.95$ use identical admissible cubes. Reference
methods can continue adjusting their displacement budgets; geometric box
splitting cannot undo a locked exclusion. A single instance and a finite
budget do not establish an asymptotic or general performance ordering}
\label{fig:bench-out}
\end{figure}

Table~\ref{tab:benchmark} records the final errors and first hitting times.
\begin{table}[htbp]
\caption{Matched quadratic experiments. Relative argument error after $200$
gradients and first iteration reaching $10^{-8}$. A dash means the threshold
was not reached within the budget; values close to $10^{-16}$ are at numerical
resolution. Parameters are fixed, not tuned to select a winner.}
\label{tab:benchmark}
\begin{tabularx}{\textwidth}{@{}lLrr@{}}
\toprule
Instance & Method & Final error & First $k$\\
\midrule
Diagonal & \CUBE, $\beta=0.5$ & $1.87\times10^{-16}$ & 27\\
Diagonal & Gradient descent & $5.51\times10^{-1}$ & --\\
Diagonal & Adam & $2.72\times10^{-5}$ & --\\
Diagonal & iRprop$^-$ & $2.01\times10^{-14}$ & 84\\
Diagonal & signGD & $1.44\times10^{-2}$ & --\\
Dense & \CUBE, $\beta=0.5$ & $1.51\times10^{-1}$ & --\\
Dense & \CUBE, $\beta=0.95$ & $3.51\times10^{-5}$ & --\\
Dense & Gradient descent & $1.94\times10^{-16}$ & 67\\
Dense & Adam & $2.45\times10^{-5}$ & --\\
Dense & iRprop$^-$ & $2.03\times10^{-14}$ & 90\\
\bottomrule
\end{tabularx}
\end{table}

\subsection{Certificates versus observed success}
The random-matrix experiment in Appendix~\ref{app:experiments} reports
unconditional frequencies of the certificate event and of finite-budget
success from a random start, with nominal $95\%$ Wilson intervals. Both
variants there receive the same $260$-gradient budget and are judged by the
same external error. Their different frequencies illustrate the quantifier
in Theorem~\ref{th:quad-iff}, not a failure of it.

On the discrete Laplacian, the computed aspect agrees with the sine vector
and both box splitting and Jacobi have dimension-dependent iteration counts.
The comparison uses the same initial point and weighted stopping criterion.
For separable sign-consistent objectives, the iteration guarantee is
independent of dimension, but the work per iteration is not: writing and
updating the sign vector already costs $O(n)$. Appendix~\ref{app:experiments}
also separates the exact smoothed oracle from its sampled approximation and
illustrates why a maximum-norm voting bound cannot omit dimension.

\section{Discussion and conclusions}\label{sec:discussion}

The speed of greedy splitting comes from the information in a trustworthy
sign, not from a large decrease in the objective. The same irreversibility
that gives optimal localisation on the exact class creates sharp failure
outside it. A useful analysis must therefore specify what certifies a sign,
how much uncertainty the overlap can absorb, and what information the oracle
actually supplies. The scalar envelope answers the worst-case question;
the matrix criterion makes it computable for quadratics; adaptive half-sizes
and the capped noise limit explain how geometric protection changes when
uncertainty is local and persistent.

Several boundaries of the theory remain meaningful.
\paragraph{Oracle complexity beyond the exact class.}\label{op:lower-msc}
The factor $(1-\theta)^{-1}$ in the safe overlapping rule is an upper bound
for this algorithm, not a proved minimax price on $\MSC_w(\theta)$.
Determining that price requires lower bounds for algorithms allowed to retain
multiple candidate regions or revisit previously rejected ones, with the
oracle and the norm fixed in advance.

\paragraph{Sharpness under additional geometry.}\label{op:envelope}
The envelope is attained when dimension and ties may be chosen. The exact
worst-case residual in a prescribed low dimension, under a prescribed tie
rule, or within strongly convex quadratics is a more restrictive problem.
The convex smooth potential used here does not settle all of these variants.
Likewise, minimising $\theta_w$ is not the same as minimising total localisation
cost once the initial covering radius and matrix setup are included.

\paragraph{Learning a certificate and analysing typical starts.}\label{op:phase}
Adaptive boxes use a known valid matrix bound; learning that bound from noisy
local information without losing retention would require a joint estimation
and localisation argument. A random-start theory must specify the matrix
ensemble, target distribution, accuracy and budget. The existence of an
adversarial start and the frequency of observed successful starts are distinct
objects, and neither can replace the other.

For applications, the resulting decision is concrete. Use halving when
coordinate signs are reliable, overlap when a certified defect is available,
and adaptive half-sizes when coordinatewise coupling bounds are known.
Account separately for the cost and bias of comparisons, averaging and
restarts. When these hypotheses cannot be justified, shrinking a single box
is an exploratory procedure, not a global-optimality certificate.

This study synthesises more than a decade of research experience of the
scientific group at the Matrosov Institute for System Dynamics and Control
Theory, Siberian Branch of the Russian Academy of Sciences. The practical
effectiveness of methods of the type considered here has long attracted the
authors' attention. The present article is an attempt to explain those
observations through theory: to identify the structures that make inexpensive
sign-based decisions reliable, and to distinguish them from situations in
which practical success cannot be promoted to a universal guarantee.

\begin{acknowledgement}
The research was supported by Russian Science Foundation (project No. 21-71-30005-$\mathrm{\Pi}$).

\end{acknowledgement}

\section*{Statements and declarations}
\paragraph{Funding.}
The work was supported by the basic research programmes of the Matrosov
Institute for System Dynamics and Control Theory, Siberian Branch of the
Russian Academy of Sciences, and of the Moscow
Institute of Physics and Technology.
\paragraph{Data and code availability.}
Online Resource~1 contains the code, random seeds, instance data, figure data
and verification outputs used in the computational study. The experiments
use synthetic data only.
\paragraph{Use of generative AI.}
During the preparation of this work, the authors used modern generative AI
models to assist with mathematical derivations, implementation of computational
checks, visualisation, and manuscript preparation. All mathematical results,
proofs, and numerical findings were verified by the authors. The authors
critically reviewed the AI-assisted material and take full responsibility for
the accuracy, integrity, and conclusions of the article. The numerical plots
were generated programmatically from the supplied data; the explanatory
diagrams were constructed from the mathematical formulas.

\FloatBarrier
\label{mainbodyend}
\bibliographystyle{spmpsci}
\bibliography{references}

\label{referencesend}
\clearpage
\begin{appendix}
\normalsize

\section{Exact localisation: proofs and examples}\label{app:exact}
\subsection*{Proof of Lemma~\ref{lem:master}}
\begin{proof}
(i) Suppose the step is correct or $e_{k,i}=0$. If
$|e_{k,i}|\ge(1-\beta)r_{k,i}$, then $|e_{k+1,i}|=|e_{k,i}|-(1-\beta)r_{k,i}$
and $d_{k+1,i}=\bigl[|e_{k,i}|-r_{k,i}\bigr]_+=d_{k,i}$; if
$|e_{k,i}|<(1-\beta)r_{k,i}$, then $|e_{k+1,i}|\le(1-\beta)r_{k,i}\le\beta
r_{k,i}=r_{k+1,i}$, so $d_{k+1,i}=0$.
(ii) Here $|e_{k+1,i}|=|e_{k,i}|+(1-\beta)r_{k,i}$, whence
$d_{k+1,i}=\bigl[|e_{k,i}|-(2\beta-1)r_{k,i}\bigr]_+$. If
$|e_{k,i}|\ge r_{k,i}$ the increment equals $2(1-\beta)r_{k,i}$; otherwise
$d_{k,i}=0$, $d_{k+1,i}\le|e_{k,i}|$ and simultaneously
$d_{k+1,i}\le2(1-\beta)r_{k,i}$.
(iii) From $x^*\in\Bx(c_0,r_0w)$ we get $d_{0,i}=0$; by (i)--(ii) the quantity
$d_{k,i}$ grows only at wrong steps, each time by at most the stated minimum.
Finally $|e_{m,i}|\le d_{m,i}+r_{m,i}$ and $r_{m,i}\to0$.\qed
\end{proof}

\begin{proposition}[relation to coordinatewise unimodality]
\label{prop:csc-unimodal}
If $f\in\CSC(x^*,B)$, then for every $i$ and every fixed $x_{-i}$ the section
$t\mapsto f(t,x_{-i})$ decreases strictly for $t<x^*_i$ and increases strictly
for $t>x^*_i$ within the section of the box; in particular $x^*$ is the unique
minimiser of $f$ on $B$.

The converse fails: strict unimodality of the sections does \emph{not} imply
$f\in\CSC(x^*,B)$. A counterexample is
$g(t)=\frac{t^4}{4}-\frac{2t^3}{3}+\frac{t^2}{2}$ with $g'(t)=t(t-1)^2$: the
function decreases strictly for $t<0$, increases strictly for $t>0$ and has a
unique minimiser, yet $g'(1)=0$ with $1\ne0=x^*$, so~\eqref{eq:csc} fails. The
two classes coincide exactly when the derivatives of the sections do not vanish
away from the minimiser.
\end{proposition}
\begin{proof}
The sign of the derivative of a section coincides with the sign of $t-x^*_i$,
whence monotonicity along each ray; moving coordinatewise from any $x\ne x^*$ to
$x^*$ decreases the value strictly. The converse is the displayed example; its
relevance is shown by Example~\ref{ex:tie}, where a vanishing derivative lets an
admissible tie-breaking rule send the method the wrong way.\qed
\end{proof}

\begin{example}[the price of one tie]\label{ex:tie}
Put
\[
  F(t)=g(1-t),\qquad g(s)=\frac{s^4}{4}-\frac{2s^3}{3}+\frac{s^2}{2},
  \qquad\text{so that}\qquad F'(t)=-(1-t)\,t^2 .
\]
Then $F$ decreases strictly on $(-\infty,1)$ and increases strictly on
$(1,+\infty)$: it is strictly unimodal with the unique minimiser $x^*=1$.
However $F'(0)=0$, i.e. at the point $t=0$, different from the minimiser, the
derivative is degenerate; the strict condition~\eqref{eq:csc} fails, only
the non-strict one holds.

Take $n=1$, $\beta=1/2$, $c_0=0$, $r_0=1$, so that $x^*\in[c_0-r_0,c_0+r_0]$.
The tie-breaking rule may choose $s_0=+1$; then $c_1=c_0-\tfrac12s_0=-\tfrac12$
and $|c_1-x^*|=\tfrac32>r_1=\tfrac12$: the solution is discarded at the first
step. Afterwards $F'(c_k)<0$ for $c_k<1$, all signs are correct and the centres
increase monotonically: $c_2=-\tfrac14$, $c_3=-\tfrac18$, \dots,
$c_k\to0$. The final error equals exactly $1=r_0$: the method has made no
progress at all. Note also that appealing to ``the set of ties has measure
zero'' is not legitimate in general: for a function with a flat piece that set
has positive measure.
\end{example}

\subsection*{Proof of Theorem~\ref{th:exact}}
\begin{proof}
\emph{Sufficiency.} Repeat the induction of Theorem~\ref{th:csc}, noting that
if $x^*\in\Bx(c_k,r_k)\subseteq B$ then $c_k\in C_w(B,x^*)$ (take $r=R_k$), so
the sign condition applies at every point of the trajectory; and by Lemma~\ref{lem:nested}
$\Bx(c_{k+1},r_{k+1})\subseteq\Bx(c_k,r_k)\subseteq B$.

\emph{Necessity.} Suppose $\partial_if(\bar c)(\bar c_i-x^*_i)\le0$ for some
$\bar c\in C_w(B,x^*)$ and some $i$ with $\bar c_i\ne x^*_i$.
By~\eqref{eq:Cw} there is $r$ with $x^*\in\Bx(\bar c,rw)\subseteq B$; take this
box as the starting one and $c_0=\bar c$. Then $e_{0,i}\ne0$. If
$\partial_if(\bar c)\ne0$, then $\sgn\partial_if(\bar c)=\sgn(e_{0,i})$, so
$s_{0,i}=\sgn(e_{0,i})$ and the first step is wrong in coordinate $i$; if
$\partial_if(\bar c)=0$, an admissible tie-breaking rule may choose
$s_{0,i}=\sgn(e_{0,i})$, again a wrong step. In both cases
Lemma~\ref{lem:lock} with $\beta=1/2$ gives
$\lim_k|e_{k,i}|\ge|e_{0,i}|>0$, contradicting convergence.\qed
\end{proof}

\begin{example}[locality matters]\label{ex:local}
Let $n=1$, $w=1$, $B=[-1,1]$, $x^*=0.9$. Then $C_1(B,x^*)=[-0.05,0.95]$: a
centre $c<-0.05$ cannot be the centre of a segment that contains $0.9$ and lies
in $B$. The function
\[
  f(x)=(x-0.9)^2+
  \begin{cases}
    40\,(x+0.9)^2(x+0.1)^2, & -0.9\le x\le-0.1,\\
    0,&\text{otherwise},
  \end{cases}
\]
belongs to $C^1(B)$, has the unique minimiser $0.9$ and satisfies the reachable-centre sign condition
of Theorem~\ref{th:exact}, so the method converges from every admissible
segment. Yet $f'(-0.7)=0.64>0$ while $-0.7<x^*$: the strict sign condition fails
on $B$. Thus~\eqref{eq:csc} on all of $B$ is sufficient but not necessary; it is
necessary exactly on the reachable centres.
\end{example}

\begin{proposition}[closure properties]\label{prop:algebra}
Let $f,g\in\CSC(x^*,B)$. Then:
\begin{enumerate}[leftmargin=*,label=\textup{(\roman*)},itemsep=1pt]
\item $\alpha f+\gamma g\in\CSC(x^*,B)$ for $\alpha,\gamma\ge0$,
  $\alpha+\gamma>0$;
\item $\varphi\circ f\in\CSC(x^*,B)$ for every differentiable $\varphi$ with
  $\varphi'>0$ on $f(B)$;
\item let $f_1,\dots,f_n\colon\R\to\R$ be differentiable with
  $f_i'(t)(t-x^*_i)>0$ for $t\ne x^*_i$, and let $\Phi\colon\R^n\to\R$ be
  differentiable with $\partial_j\Phi>0$; then
  $F(x)=\Phi\bigl(f_1(x_1),\dots,f_n(x_n)\bigr)\in\CSC(x^*,B)$;
\item $\max\{f,g\}\in\CSC(x^*,B)$ at every point of differentiability;
\item if $f_m\in\CSC^{\,\circ}(x^*,B)$, $f_m\to f$ pointwise and the limit $f$
  is differentiable, then $f\in\CSC^{\,\circ}(x^*,B)$.
\end{enumerate}
\end{proposition}
\begin{proof}
(i) A sum of quantities of the same sign, at least one of them strictly
positive. (ii) See Proposition~\ref{lem:invariance}. (iii)
$\partial_iF=\partial_i\Phi\cdot f_i'(x_i)$ with $\partial_i\Phi>0$ and
$\sgn f_i'(x_i)=\sgn(x_i-x^*_i)$. (iv) At a point of differentiability the
gradient coincides with that of the active function, and both lie in $\CSC$.
(v) Derivatives are not inherited from pointwise convergence, so we argue
through monotonicity: by Proposition~\ref{prop:csc-unimodal} each section of
$f_m$ is non-increasing to the left of $x^*_i$ and non-decreasing to the right;
a pointwise limit of monotone functions is monotone in the same sense, and a
differentiable monotone function has a derivative of the required (non-strict)
sign.\qed
\end{proof}

\subsection*{Proof of Corollary~\ref{cor:csc-f}}
\begin{proof}
The weighted localisation bound gives
$\|c_k-x^*\|_2\le R_k\|w\|_2$.
Stationarity and Lipschitz continuity imply
$\|\nabla f(c_k)\|_2\le L\|c_k-x^*\|_2$.
The descent lemma with base point $x^*$ gives
$f(c_k)-f(x^*)\le L\|c_k-x^*\|_2^2/2$; nonnegativity follows from minimality.
These are the two assertions.\qed
\end{proof}

\subsection*{Proof of Theorem~\ref{th:subgrad}}
\begin{proof}
The coordinatewise induction of Theorem~\ref{th:csc} applies verbatim: for
$c_{k,i}\ne x^*_i$ condition~\eqref{eq:subcsc} gives a correct, non-degenerate
sign, while for $c_{k,i}=x^*_i$ both directions are safe. Differentiability is
never used.\qed
\end{proof}

\begin{example}[$\ell_1$ in, $\ell_\infty$ out]\label{ex:l1}
For $f(x)=\sum_ia_i|x_i-x^*_i|$ with $a_i>0$ every subgradient has
$g_i=a_i\sgn(x_i-x^*_i)$ whenever $x_i\ne x^*_i$, so~\eqref{eq:subcsc} holds and
Theorem~\ref{th:subgrad} applies. By contrast, for $f(x)=\ninf{x-x^*}$ at the
point $x=(0.9,0.8)$ (with $x^*=0$) the subdifferential consists of the vector
$(1,0)$: the inactive coordinate contributes nothing, the condition fails, and
an admissible negative tie in that coordinate locks the error $0.8$ by Lemma~\ref{lem:lock}. The difference between
$\ell_1$ and $\ell_\infty$ here is structural and has nothing to do with
smoothness.
\end{example}

\subsection*{Proof of Theorem~\ref{th:lower}}
\begin{proof}
Consider $f_z(x)=\|x-z\|_2^2/2$, with
$z\in Q_0:=\Bx(c_0,r_0w)$. Its sign oracle compares each query coordinate
with $z_i$. Normalise the first target coordinate as
$\zeta=(z_1-c_{0,1})/w_1\in(-r_0,r_0)$, and keep all other target coordinates
fixed. Start with the open interval $I_0=(-r_0,r_0)$ of possible $\zeta$.
At each query the adversary keeps the longer of the two portions cut by the
normalised query coordinate, returning the corresponding nonzero sign. If the
query lies outside $I_{j-1}$, it keeps that entire interval. Thus
$|I_j|\ge |I_{j-1}|/2$, and all points of $I_j$ remain consistent with the
history. Answers in the other coordinates are the fixed target's actual signs.
After $k$ queries, $|I_k|\ge2^{1-k}r_0$. Any reported first coordinate has
worst-case weighted error, in the supremum over this open interval, at least
$|I_k|/2\ge2^{-k}r_0$. Equality answers do not improve this worst case, because
the adversary never needs to return zero in the uncertain coordinate.\qed
\end{proof}

\begin{remark}[randomised expected error]\label{rem:yao}
For a lower bound on the expected error of a randomised algorithm, put a uniform
prior on the normalised first target coordinate in $(-r_0,r_0)$ and fix the
others. This is a prior on a one-dimensional slice of $Q_0$, not the uniform
prior on the full box. For a deterministic algorithm, the nonzero answers
partition the slice into at most $2^k$ intervals, with lengths $\ell_j$ summing
to $L=2r_0$. The output is constant on each leaf, and its integrated absolute
error on an interval is at least $\ell_j^2/4$. Hence
\[
 \E\,\frac{|\widehat x_1-z_1|}{w_1}
 \ge\frac{1}{4L}\sum_j\ell_j^2
 \ge\frac{L}{4\,2^k}=2^{-(k+1)}r_0.
\]
The same prior lower bound holds after averaging over the algorithm's random
seed, and therefore lower-bounds its worst-input expected weighted error.
This is the elementary distributional argument underlying Yao's
principle~\cite{Yao1977}. It preserves the rate, not the deterministic constant.
\end{remark}

\section{The defect envelope and its realisation}\label{app:envelope}\label{app:proofs}
\subsection*{Proof of Proposition~\ref{prop:clipped-path}}
Write $a_i=|x_i-x_i^*|/w_i$ and $M=\max_i a_i>0$.
The path $\gamma$ in the main text is continuous and differentiable except at
finitely many values $a_i$. On every open interval between successive distinct
values in $\{0,a_1,\dots,a_n\}$, the moving set is
$I(t)=\{i:a_i>t\}$ and is nonempty. For $0<t<M$,
$\|\gamma(t)-x^*\|_{\infty,w}=t$, and every $i\in I(t)$ has normalised error
exactly $t>\theta t$. Hence
\[
 \frac{d}{dt}f(\gamma(t))=
 \sum_{i\in I(t)}w_i\sgn(x_i-x_i^*)\,\partial_if(\gamma(t))>0.
\]
The mean value theorem gives strict increase on each closed piece; continuity
joins the pieces. Thus $f(x)>f(x^*)$. The path lies coordinatewise between
$x^*$ and $x$, hence inside the box. No continuity of the derivative beyond
that required for differentiability is used.\qed

\subsection*{Proof of Theorem~\ref{th:theta}}
\begin{proof}
Pass to the normalised variables $u_{k,i}=(c_{k,i}-x^*_i)/w_i$, so that
$M_k=\ninf{u_k}$ and the step length in each coordinate equals $h_k=R_k/2$ with
$R_k=2^{-k}r_0$. If $|u_{k,i}|>\theta M_k$, then by~\eqref{eq:msc} the sign is
correct and
$|u_{k+1,i}|=\bigl||u_{k,i}|-h_k\bigr|\le\max\{M_k-h_k,\ h_k\}$. For the
remaining coordinates $|u_{k+1,i}|\le\theta M_k+h_k$ whatever the sign. Hence
\begin{equation}\label{eq:scalar-halving}
  M_{k+1}\le\max\bigl\{M_k-\tfrac{R_k}{2},\ \theta M_k+\tfrac{R_k}{2}\bigr\}.
\end{equation}
From $M_0\le r_0$ we get $M_1\le\max\{r_0/2,\ \theta r_0+r_0/2\}=R_1+\theta r_0$.
Assume $M_k\le R_k+\theta r_0$ for some $k\ge1$. Since $R_k\le r_0/2$ and
$\theta\le1/2$, this gives $M_k\le r_0$. Then the first argument of the maximum
in~\eqref{eq:scalar-halving} is at most $R_k+\theta r_0-\tfrac{R_k}{2}=R_{k+1}+\theta
r_0$, and the second at most $\theta r_0+\tfrac{R_k}{2}=R_{k+1}+\theta r_0$. The
induction is complete; the limit exists because the step lengths are
summable.\qed
\end{proof}

\begin{proposition}[the constant $1$ is sharp]\label{prop:tight}
For every $\theta\in(0,1/2]$ and every $\eta\in(0,\theta)$ there are a positive
definite quadratic function of two variables with $\theta_{\mathbf1}(H)=\theta$
(Definition~\ref{def:defect}) and a starting cube (depending on $\eta$) for
which $\lim_kM_k\ge\eta r_0$. Namely, for
$H_\theta=\bigl(\begin{smallmatrix}1&\theta\\\theta&1\end{smallmatrix}\bigr)$,
$x^*=0$, $w=\mathbf1$ and $c_0=r_0(\eta,-1)$ the first step is wrong in the
first coordinate. Consequently no universal bound
$\lim_kM_k\le\mathrm{const}\cdot\theta r_0$ with a constant smaller than one can
hold.
\end{proposition}
\begin{proof}
$\partial_1f(c_0)=r_0(\eta-\theta)<0$ while $c_{0,1}-x^*_1=\eta r_0>0$; hence
the first step is wrong in coordinate~1 and Lemma~\ref{lem:lock} with
$\beta=1/2$ gives $\lim_k|e_{k,1}|\ge\eta r_0$. The matrix $H_\theta$ is
positive definite for $\theta<1$ and its defect equals $\theta$. Letting
$\eta\to\theta-0$ rules out any constant below one.\qed
\end{proof}

\subsection*{Analysis of the scalar envelope (Proposition~\ref{prop:envelope})}

Let $\theta>2\beta-1$ and $q_k=b_k/R_k$ with $R_k=\beta^kr_0$. The second
argument of the maximum in~\eqref{eq:envelope} is at least the first exactly
when
\[
 \begin{aligned}
 \theta b_k+(1-\beta)R_k&\ge b_k-(1-\beta)R_k\\
 &\iff 2(1-\beta)R_k\ge(1-\theta)b_k\\
 &\iff q_k\le\frac{2(1-\beta)}{1-\theta}=:q^\sharp.
 \end{aligned}
\]
While $q_k\le q^\sharp$ we have $b_{k+1}=\theta b_k+(1-\beta)R_k$, that is,
\[
  q_{k+1}=\frac{\theta q_k+1-\beta}{\beta},
\]
an increasing affine recursion with fixed point $q^\star=(1-\beta)/(\beta-\theta)$
when $\theta<\beta$. The inequality $q^\star>q^\sharp$ is equivalent to
$1+\theta>2\beta$, i.e. precisely to the assumption $\theta>2\beta-1$; hence
$q_k$ increases strictly towards $q^\star$ and crosses $q^\sharp$ after finitely
many steps. For $\theta\ge\beta$ the sequence $q_k$ grows without bound, so the
threshold is crossed a fortiori.

Let $J$ be the first index with $q_J\ge q^\sharp$. Then
$b_{J+1}=b_J-(1-\beta)R_J$ and
\[
  q_{J+1}=\frac{q_J-1+\beta}{\beta}\ \ge\ q_J
  \iff q_J\ \ge\ 1 ,
\]
while $q^\sharp\ge1$ is equivalent to $\theta\ge2\beta-1$, which holds.
Therefore the first argument dominates from step $J$ onwards and
\[
  b_\infty=b_J-\sum_{k\ge J}(1-\beta)R_k=b_J-R_J .
\]
If $\theta\le2\beta-1$, a direct check gives $b_k=R_k$: indeed
$\theta R_k+(1-\beta)R_k\le(2\beta-1)R_k+(1-\beta)R_k=\beta R_k=R_{k+1}$ and
$R_k-(1-\beta)R_k=R_{k+1}$. \qed

\subsection*{Sharpness of the envelope (Theorem~\ref{th:attain})}

Throughout $r_0=1$, $w=\mathbf1$, $x^*=0$, $R_k=\beta^k$, $h_k=(1-\beta)R_k$ and
$b_k$ is given by~\eqref{eq:envelope}.

\paragraph{The potential.} Let $Z_1,\dots,Z_n$ be independent with the common
density
\[
  p(z)=C_\theta\exp\Bigl[-\frac{1}{(z-\theta)(1-z)}\Bigr]
  \quad\text{for }\theta<z<1,\qquad p(z)=0\text{ otherwise},
\]
which is $C^\infty$, positive inside the interval and flat at both endpoints,
and let $G$ and $F=\psi\circ G$ be as in~\eqref{eq:G-potential}. As an
expectation of maxima of convex quadratics, $G$ is convex and positively
homogeneous of degree two, and
$\ninf{x}^2\le G(x)\le\theta^{-2}\ninf{x}^2$.

For $x\ne0$ the maximising index is almost surely unique, so differentiation
under the expectation gives
\[
  \partial_iG(x)=2x_i\,
  \E\Bigl[Z_i^{-2}\,\mathbf1\Bigl\{\frac{|x_i|}{Z_i}
    >\max_{j\ne i}\frac{|x_j|}{Z_j}\Bigr\}\Bigr],
\]
which is legitimate because $Z_i\ge\theta>0$ and the difference quotients are
locally dominated. Put $M=\ninf{x}$. If $|x_i|\le\theta M$ then
$|x_i|/Z_i\le M$, whereas a coordinate $j$ with $|x_j|=M$ has $|x_j|/Z_j>M$;
hence $i$ never maximises and $\partial_iG(x)=0$. If $|x_i|>\theta M$, then on
the positive-probability event $\{Z_i$ close to $\theta$, $Z_j$ close to $1\}$
the index $i$ does maximise, so $\partial_iG(x)x_i>0$. This
is~\eqref{eq:exact-sign-field}, i.e. $G\in\MSC_{\mathbf1}(\theta,0,\Rn)$ with a
sign field that vanishes identically inside the band.

\paragraph{Regularity.} With the smooth tail $Q(s)=\Prob\{Z\ge s\}$, which
equals $1$ for $s\le\theta$ and $0$ for $s\ge1$, the layer-cake formula gives
\[
  G(x)=\int_0^\infty\Bigl[1-\prod_{i=1}^nQ\Bigl(\frac{|x_i|}{\sqrt t}\Bigr)\Bigr]dt .
\]
In a small neighbourhood of any $x\ne0$ the integrand equals $1$ for all small
$t$ and $0$ for all large $t$, with bounds uniform in the neighbourhood, and is
$C^\infty$ in $x$ on the intermediate compact range; the modulus at $x_i=0$ is
harmless because $Q$ is locally constant there. Hence
$G\in C^\infty(\Rn\setminus\{0\})$. Degree-two homogeneity alone would not give
smoothness at the origin, which is why the flat outer function
$\psi(s)=\int_0^se^{-1/t}dt$ is used: $\psi'(s)=e^{-1/s}>0$ and
$\psi''(s)=e^{-1/s}/s^2>0$, so $F=\psi\circ G$ is convex, has the same sign
field away from the origin, and, since $D^mG(x)=O(\|x\|^{2-m})$ while all
derivatives of $\psi$ vanish faster than any power as $s\downarrow0$, all
derivatives of $F$ extend by zero to the origin. Thus $F\in C^\infty(\Rn)$, and
$F$ has the unique minimiser $0$. Finally, for $x\ne0$
\[
  \nabla^2F(x)=e^{-1/G(x)}\nabla^2G(x)
   +\frac{e^{-1/G(x)}}{G(x)^2}\nabla G(x)\nabla G(x)^{\!\top},
\]
where $\nabla^2G$ is bounded off the origin (degree-zero homogeneity plus
compactness of the unit sphere) and $\|\nabla G(x)\|^2\le CG(x)$, so the second
term is bounded by $Ce^{-1/G}/G\le C/e$; at the origin the Hessian is zero.
Hence $\nabla F$ is globally Lipschitz.

\begin{lemma}[the full interval of one-step preimages]\label{lem:interval}
Let $b\ge h\ge0$. Suppose that on $[-b,b]$ the sign is forced to be $\sgn t$
when $|t|>\theta b$ and may be either sign when $|t|\le\theta b$. Then the set
of all attainable values $t-hs$ is exactly $[-B,B]$ with
$B=\max\{b-h,\ \theta b+h\}$.
\end{lemma}

\begin{proof}
Combining the outer region $t\in[-b,-\theta b)$, where $s=-1$, with the band,
where $s=-1$ is allowed, produces the interval $[-b+h,\ \theta b+h]$; the
mirrored pair produces $[-\theta b-h,\ b-h]$. Both contain $0$ because $h\le b$,
and their union is $[-B,B]$. The endpoints of the band are included since a
vanishing derivative admits either sign.\qed
\end{proof}

\paragraph{Completion of the proof of Theorem~\ref{th:attain}.}
The recursion gives $b_k\ge R_k\ge h_k$, so Lemma~\ref{lem:interval} applies at
every step. For each $j=0,\dots,J$ construct backwards a scalar trajectory
starting in $[-1,1]$ that equals $+b_j$ at time $j$ and obeys the rule with the
thresholds $\theta b_k$ at the earlier steps; after time $j$ continue it with
any admissible signs. Take these $J+1$ trajectories as the coordinates of one
vector. At every time $k\le J$ all coordinates lie in $[-b_k,b_k]$ by
Lemma~\ref{lem:interval}, and the $k$-th coordinate equals $b_k$; hence the
\emph{actual} maximum is $b_k$, so the thresholds used in the construction agree
with the true vector state, and by~\eqref{eq:exact-sign-field} every selected
sign is either the forced sign of a nonzero derivative of $F$ or an admissible
resolution of a vanishing one. This is the gradient of one fixed function, not
an arbitrary non-integrable vector field.

At time $J$ the distinguished coordinate equals $b_J>R_J$. From then on it stays
maximal, its sign is therefore forced, and it decreases by the whole remaining
budget:
$b_J-\sum_{\ell=J}^{k-1}h_\ell=b_J-R_J+R_k=b_k$. The general upper bound of
Proposition~\ref{prop:envelope} prevents the other coordinates from exceeding
$b_k$, so $M_k=b_k$ for all $k$ and $M_\infty=b_J-R_J$. For
$\theta\le2\beta-1$ the one-dimensional example $f(t)=t^2/2$, $c_0=1$ gives
$c_k=R_k$ directly. \qed

\paragraph{Constructing the starting point.} Given the required value $y$ at the
next step, only two candidates need to be tested: $t=y+h_k$ with $s=+1$ and
$t=y-h_k$ with $s=-1$; keep the one with $|t|\le b_k$ and either
$|t|\le\theta b_k$ or $st>0$. Lemma~\ref{lem:interval} guarantees that a
candidate exists. With rational parameters all tests are exact; this is the
routine \texttt{exact\_envelope\_witness} of the accompanying code, A particularly simple choice gives Example~\ref{ex:witness}.

\section{Matrix certificates and adaptive geometry}\label{app:matrix}
\subsection*{Proof of Theorem~\ref{th:hess}}
\begin{proof}
Let $u=x-x^*$; by the Newton--Leibniz formula $\nabla f(x)=\bar H(x)u$. Let $i$
be such that $|u_i|/w_i>\theta M$ with $M=\ninfw{u}$. Using $|u_j|\le w_jM$ and
the definition of $\theta_w$,
\begin{align*}
  \partial_if(x)\,u_i
  &=\bar H_{ii}u_i^2+\sum_{j\ne i}\bar H_{ij}u_ju_i
   \ \ge\ \bar H_{ii}u_i^2-|u_i|M\sum_{j\ne i}|\bar H_{ij}|w_j\\
  &\ \ge\ \bar H_{ii}|u_i|\bigl(|u_i|-\theta w_iM\bigr)\ >\ 0. \qquad\qed
\end{align*}
\end{proof}

\subsection*{Attainability of the minimum in Theorem~\ref{th:perron}}

Write $D=\diag(H_{11},\dots,H_{nn})$, $B=H-D$ and $A=D^{-1}|B|$, with $H$ symmetric and $H_{ii}>0$. The matrix
$S=D^{-1/2}|B|D^{-1/2}$ is symmetric, nonnegative and similar to $A$ (through
$A=D^{-1/2}SD^{1/2}$), so $\rho(A)=\rho(S)=\lambda_{\max}(S)$.

If the graph of the nonzero entries of $|B|$ is connected, then $S$ is
irreducible and by the Perron--Frobenius theorem the eigenvalue
$\lambda_{\max}(S)$ has a strictly positive eigenvector $v$; setting
$w=D^{-1/2}v$ we obtain
\[
  \theta_w(H)=\max_i\frac{(Aw)_i}{w_i}
  =\max_i\frac{\bigl(D^{-1/2}Sv\bigr)_i}{\bigl(D^{-1/2}v\bigr)_i}
  =\lambda_{\max}(S)=\rho(A),
\]
so the infimum is attained. Conversely, fix any positive weight vector $w$. The similar matrix
$\diag(w)^{-1}A\diag(w)$ has row sums $(Aw)_i/w_i$. The spectral radius is
bounded by the maximum absolute row sum, hence $\rho(A)\le\theta_w(H)$.

If the graph is disconnected, then after a permutation $S$ splits into a direct
sum of blocks $S^{(1)},\dots,S^{(m)}$ corresponding to the connected components,
together with isolated coordinates (zero $1\times1$ blocks). On each block of
size $\ge2$ take its own Perron vector $v^{(j)}>0$, and on the isolated
coordinates an arbitrary positive weight. The resulting $w>0$ gives
\[
  \theta_w(H)=\max_j\rho\bigl(S^{(j)}\bigr)=\rho(S)=\rho(A),
\]
since the spectrum of a direct sum is the union of the spectra and the rows of
$A$ corresponding to isolated coordinates vanish. Thus the minimum is attained
in the reducible case as well; note, however, that the glued vector is an
eigenvector of each block separately and, when the blocks have different
spectral radii, not of $A$. \qed

\subsection*{Proof of Theorem~\ref{th:quad-iff}}
\begin{proof}
\emph{Sufficiency.} If $\theta:=\theta_w(H)\le2\beta-1<1$, then by
Theorem~\ref{th:hess} $f\in\MSC_w(\theta,x^*,\Rn)$, and
Theorem~\ref{th:beta} (with $\lambda=0$) gives $M_k\le\beta^kr_0$.
Moreover, $\diag(w)H\diag(w)$ is symmetric, has positive diagonal and is
strictly diagonally dominant; Gershgorin's theorem makes it positive definite.
Congruence therefore also makes $H$ positive definite.

\emph{Necessity.} Put $a:=2\beta-1$ and suppose that for some row $i$
\[
  \theta_i:=\frac{\sum_{j\ne i}|H_{ij}|w_j}{H_{ii}w_i}>a .
\]
Choose $t$ with $a<t<\min\{\theta_i,1\}$ and set $c_0=x^*+u$, where
\[
  u_i=t\,r_0w_i,\qquad u_j=-\sgn(H_{ij})\,r_0w_j\quad(j\ne i)
\]
(for $H_{ij}=0$ either sign will do). Then $\ninfw{u}\le r_0$, i.e.
$x^*\in\Bx(c_0,r_0w)$, and
\[
  \partial_if(c_0)=(Hu)_i
  =H_{ii}t\,r_0w_i-\sum_{j\ne i}|H_{ij}|\,r_0w_j
  =H_{ii}r_0w_i\,(t-\theta_i)<0,
\]
whereas $u_i>0$. Hence the first step is wrong in coordinate $i$ (there is no
tie here, so the example does not depend on the tie-breaking rule), and by
Lemma~\ref{lem:lock}
\[
  \lim_{k\to\infty}\frac{|c_{k,i}-x^*_i|}{w_i}\ \ge\ t\,r_0-a\,r_0=(t-a)r_0>0 .
  \eqno\qed
\]
\end{proof}

\subsection*{Proof of Corollary~\ref{cor:hmat}}
\begin{proof}
If $\theta_{\mathrm{opt}}<1$ the minimum is attained at some $w^*>0$ by
Theorem~\ref{th:perron}, and~\eqref{eq:quad-exact} holds for
$\beta\ge(1+\theta_{\mathrm{opt}})/2$. Conversely, if $\theta_w(H)\ge1$ for all
$w>0$, then $\theta_w(H)>2\beta-1$ for every $\beta<1$, and
Theorem~\ref{th:quad-iff} forbids universal convergence.\qed
\end{proof}

\subsection*{Jacobi convergence does not imply the sign certificate}
Let $C$ be the symmetric $4\times4$ matrix with zero diagonal, all
upper-triangular entries $+1$ except $C_{12}=-1$. Its eigenvalues are
$\{-\sqrt5,-1,1,\sqrt5\}$. Thus $H=I+(2/5)C$ is positive definite and its
Jacobi iteration matrix $-(2/5)C$ has spectral radius $2\sqrt5/5<1$.
However, $A=(2/5)|C|$ has constant row sum $6/5$, so no fixed aspect and
$\beta<1$ provide universal sign-splitting convergence. The obstruction is
absolute cross-coupling; cancellation useful to Jacobi is lost by taking signs.

\begin{counterexample}[locked error on a strongly convex quadratic]
\label{cex:2x2}
Let $n=2$, $H=\bigl(\begin{smallmatrix}1&1\\1&2\end{smallmatrix}\bigr)$,
$x^*=0$, $c_0=(0.9,\,-1)$, $r_0=1$, $\beta=1/2$, $w=\mathbf1$, as in
Sect.~\ref{sec:early-example}. The limit is $x_\infty=(0.9,-0.5)$ with
\[
  \ninf{x_\infty-x^*}=0.9\,r_0,
  \qquad
  \frac{f(x_\infty)-f^*}{f(c_0)-f^*}=\frac{41}{101}\approx0.406 ,
\]
so the relative objective gap cannot be guaranteed below $0.406$ in this run. The cause is visible
from Theorem~\ref{th:quad-iff}: $\theta_{\mathbf1}(H)=1>2\beta-1=0$. At the same
time $\theta_{\mathrm{opt}}(H)=1/\sqrt2$ with $w^*=(1,\,1/\sqrt2)$, so the
problem \emph{is} solvable with the right box---and with the radius
$r_0=\ninfw{c_0-x^*}=\sqrt2$, since otherwise the solution is not in the box at
all (Fig.~\ref{fig:counterex}).
\end{counterexample}

\subsection*{Proof of Theorem~\ref{th:impossible}}
\begin{proof}
Write $u=c-x^*$, $\tilde u=u\oslash w$, $S=\mathbf1^{\!\top}u$ and
$v=u-\frac Sn\mathbf1$. The Hessian
$\mathbf1\mathbf1^{\!\top}+\epsilon(I-\mathbf1\mathbf1^{\!\top}/n)$ has
eigenvalues $n$ and $\epsilon>0$, so it is positive definite. Thus $\partial_if_\epsilon(c)=S+\epsilon v_i$. Let
$p=\arg\min_iw_i$ and $\gamma=2w_p/(\mathbf1^{\!\top}w)\le2/n$. Take
\[
  u_{0,j}=r_0w_j\ (j\ne p),\qquad u_{0,p}=-r_0w_p,
\]
so that $M_0=\ninf{\tilde u_0}=r_0$ (the solution is a vertex of the box) and
$S_0=r_0(\mathbf1^{\!\top}w)(1-\gamma)>0$.

Put $T=\min\{k:\ \beta^k<(1+\eta)\gamma\}$; since $(1+\eta)\gamma<1$ for
$n\ge5$, $\eta<1/4$, we have $T\ge1$. We show by induction that all signs equal
$+1$ for $k<T$. As long as this holds,
$\tilde u_{k+1}=\tilde u_k-(1-\beta)R_k\mathbf1$ and
\[
  S_k=S_0-(1-\beta^k)r_0(\mathbf1^{\!\top}w)
     =r_0(\mathbf1^{\!\top}w)\bigl(\beta^k-\gamma\bigr)
     \ \ge\ \eta\gamma\,r_0(\mathbf1^{\!\top}w)>0
  \qquad(k<T),
\]
because $\beta^k\ge(1+\eta)\gamma$. Moreover
$\ninf{v_k}\le\ninf{v_0}+r_0\ninf{w}=:V$, since at each step $v$ is shifted by
$-(1-\beta)R_k(w-\bar w\mathbf1)$ and the moduli of these shifts sum to at most
$r_0\ninf{w}$. Choosing $\epsilon<\eta\gamma r_0(\mathbf1^{\!\top}w)/V$ gives
$\partial_if_\epsilon(c_k)=S_k+\epsilon v_{k,i}>0$ for all $i$: there are no
ties, and the induction closes.

Along coordinate $p$ all these steps move \emph{away} from the target:
$\tilde u_{T,p}=-r_0-(1-\beta^T)r_0=-(2-\beta^T)r_0$. The total possible
displacement after step $T$ equals $\sum_{k\ge T}(1-\beta)R_k=R_T=\beta^Tr_0$,
whence
\[
  \lim_kM_k\ \ge\ (2-\beta^T)r_0-\beta^Tr_0=(2-2\beta^T)r_0
  \ >\ \bigl(2-2(1+\eta)\gamma\bigr)r_0 ,
\]
and $\gamma\le2/n$ gives the assertion.
No control of the signs after step $T$ is needed. The upper bound $2r_0$ follows
from $|c_{\infty,i}-x^*_i|\le|c_{0,i}-x^*_i|+\sum_k(1-\beta)R_kw_i\le2r_0w_i$.\qed
\end{proof}

\subsection*{Proof of Theorem~\ref{th:adaptive}}
\begin{proof}
Write $u=c_k-x^*$. If the sign in coordinate $i$ is wrong, or the derivative
vanishes while $u_i\ne0$, then
$H_{ii}|u_i|\le\sum_{j\ne i}|H_{ij}||u_j|\le H_{ii}(Ar_k)_i$; together with
$|u_i|\le r_{k,i}$ this gives $|u_i|\le t_{k,i}$, so after the step
$|u_i-h_{k,i}s_i|\le t_{k,i}+h_{k,i}=r_{k+1,i}$. If the sign is correct, then
$\bigl||u_i|-h_{k,i}\bigr|\le\max\{r_{k,i}-h_{k,i},h_{k,i}\}=r_{k+1,i}$, since
$h_{k,i}\le r_{k,i}/2$; the case $u_i=0$ is covered by the same estimate.
Nestedness follows from $h_k+r_{k+1}=r_k$. Finally, the map
$T(r)=\frac12\bigl(r+\min\{r,Ar\}\bigr)$ is monotone and positively homogeneous,
and $Aw\le qw$ gives $T(w)\le\frac{1+q}{2}w$; induction from $r_0\le C_0w$
yields~\eqref{eq:adapt-rate}.\qed
\end{proof}

\subsection*{One-step optimality in Theorem~\ref{th:adaptive}}
Fix a coordinate $i$, let $a=(Ar)_i$, and suppose a proposed half-size
$r'_i$ is smaller than $(r_i+\min\{r_i,a\})/2$. For $a>0$ choose
$t\in(0,r_i]$ with $t<a$ and $t>2r'_i-r_i$; this is possible by the strict
inequality. Set the centre error $u_i=t$ and
$u_j=-\sgn(H_{ij})r_j$ for $j\ne i$ (arbitrary signs when $H_{ij}=0$).
Then $(Hu)_i=H_{ii}(t-a)<0$: the sign is wrong, so the new error is
$t+r_i-r'_i>r'_i$.
If $a=0$, the assumed inequality is $r'_i<r_i/2$; use $u_i=0$ and either
tie, giving new error $r_i-r'_i>r'_i$. Conversely the half-size in
\eqref{eq:one-step-opt} is safe by the retention proof. This establishes
componentwise minimality for rules which choose $r'$ from $H,r$ before seeing
the signs. It does not compare more informative, sign-dependent rules.\qed

\subsection*{Proof of Corollary~\ref{cor:capped-limit}}
With gradient perturbation bounded by $H_{ii}\xi_i$, a wrong sign is possible
only if $|c_i-x_i^*|\le(Ar)_i+\xi_i$. The proof of retention therefore applies
with $t=\min\{r,Ar+\xi\}$; nestedness follows from $h+r'=r$.

Since $\rho(A)<1$, choose $w>0$ and $q<1$ with $Aw\le qw$; for example,
$w=(I-A)^{-1}\mathbf1$ is positive and
$q=\max_i(Aw)_i/w_i<1$. The map
$P(z)=\min\{r^{\mathrm{init}},Az+\xi\}$ maps the compact interval
$[0,r^{\mathrm{init}}]$ into itself and is a contraction in
$\|\cdot\|_{\infty,w}$ with constant at most $q$. It has a unique fixed point
$v$. In particular $v\le Av+\xi$; hence $T(v)=v$ for
$T(r)=(r+\min\{r,Ar+\xi\})/2$.
The map $T$ is monotone and $T(r)\le r$, so
$v\le r_{k+1}\le r_k\le r^{\mathrm{init}}$.

For any $v\le r\le r^{\mathrm{init}}$, put $d=r-v\ge0$.
If $v_i<r_i^{\mathrm{init}}$, the fixed-point equation gives
$v_i=(Av+\xi)_i$, and
\[
 0\le T(r)_i-v_i
 \le\tfrac12\bigl(d_i+(Ad)_i\bigr).
\]
If $v_i=r_i^{\mathrm{init}}$, then $r_i=v_i$ and the monotonicity sandwich
gives $T(r)_i=v_i$, so the same inequality holds. Induction proves
\eqref{eq:capped-rate}. Since $\rho((I+A)/2)=(1+\rho(A))/2<1$, it also proves
$r_k\to v$. Multiplying $(I-A)v\le\xi$ by the nonnegative inverse yields
$v\le (I-A)^{-1}\xi=:g$. If $r^{\mathrm{init}}\ge g$, then $P(g)=g$, so
uniqueness gives $v=g$.
Finally, the total displacement in coordinate $i$ is
$\sum_k h_{k,i}=r_i^{\mathrm{init}}-v_i<\infty$. The centres converge, and
passing to the limit in retention gives $|c_\infty-x^*|\le v$.\qed

\section{Comparisons, margins and residual floors}\label{app:comparison}
\begin{proposition}[relation between the classes]\label{prop:cmp-vs-csc}
If $f$ is differentiable and $f\in\CSC(x^*,B)$, then $f\in\CMP(x^*,B)$. The
converse fails: the function $F(t)=g(1-t)$ of Example~\ref{ex:tie} lies in
$\CMP(1,\R)$ but not in $\CSC(1,\R)$, since $F'(0)=0$. Moreover $\CMP$ contains $\sum_i|x_i-x^*_i|^p$ for every $p>0$.
This gives nondifferentiable examples for $0<p\le1$ and nonconvex examples for
$0<p<1$. Every differentiable member of $\CMP$ lies in $\CSC^\circ$; the
constant function shows that this inclusion cannot be reversed.
\end{proposition}
\begin{proof}
The first claim is Proposition~\ref{prop:csc-unimodal}; the rest is a direct
verification of monotonicity of the sections. A differentiable monotone
section has the required non-strict derivative sign, while a constant section
is not strictly monotone.\qed
\end{proof}

\subsection*{Proof of Theorem~\ref{th:probe-sharp}}
Assume retention at step $k$ and fix a coordinate; abbreviate its error by
$u=c_i-x_i^*$ and half-size by $r$. If $|u|\ge qr$, the two probes lie on the
same monotone section side, possibly with one at the target. Strict
monotonicity including that endpoint gives the correct comparison. Therefore
$|u'|=||u|-(1-\beta)r|\le\beta r$ for $1/2\le\beta<1$.
If $|u|<qr$, either answer gives
$|u'|\le |u|+(1-\beta)r<(q+1-\beta)r\le\beta r$.
This also covers equal probe values. Induction proves sufficiency, including
$q=2\beta-1$ when $q>0$.

For necessity, $q>2\beta-1$ permits a target $t$ in the open interval
$(2\beta-1,q)\subset(0,1]$. The displayed $F_t$ in the main text belongs to
$\CMP(t,[-1,1])$ and satisfies $F_t(q)>F_t(-q)$ for the stated $C$.
Thus the first centre is $-(1-\beta)$; its distance from $t$ is
$t+1-\beta$, while all later travel is at most $\beta$.
The limiting error is at least $t-(2\beta-1)>0$. Its derivative is continuous
and piecewise linear with positive slopes $2$ and $2C$, so $F_t$ is strongly
convex with globally Lipschitz derivative. This proves necessity with no
reliance on a tie at the first step.\qed

\subsection*{Proof of Proposition~\ref{prop:polarization}}
\begin{proof}
Expand the quadratic form: the terms linear in $h$
give~\eqref{eq:polarization} and the quadratic ones cancel.\qed
\end{proof}

\subsection*{Proof of Theorem~\ref{th:cmp-lower}}
\begin{proof}
For $Q_0=\Bx(c_0,r_0w)$, a binary depth-$N$ decision tree has at most $2^N$
leaves. At any fixed leaf, the targets for which its output has weighted error
at most $\eps$ occupy volume at most $(2\eps)^n\prod_iw_i$. Exact success
for all targets therefore requires
$2^N(2\eps)^n\prod_iw_i\ge(2r_0)^n\prod_iw_i$.
For a randomised algorithm of fixed budget, apply the same upper bound on the
fraction of successful targets to each fixed random seed and integrate over
seeds. If success probability is at least $1-\delta$ for every target, its
average over the uniform target distribution is at least $1-\delta$.
Consequently $(1-\delta)\le2^N(\eps/r_0)^n$, proving both inequalities.
The oracle is binary even at equality. On spherical quadratics, symmetric
coordinate probes reveal the correct side whenever the coordinate error is
nonzero, so one comparison per coordinate performs exact halving; ties at the
target are harmless. This proves the stated attainability on that subclass.\qed
\end{proof}

\subsection*{Proof of Theorem~\ref{th:inexact}}
\begin{proof}
Induction. Assume $M_k\le R_k$. A wrong step in coordinate $i$ is possible only
if $|u_{k,i}|\le\theta M_k+\eta_k\le\theta R_k+(2\beta-1-\theta)R_k
=(2\beta-1)R_k$; adding the step $(1-\beta)R_k$ gives
$|u_{k+1,i}|\le\beta R_k=R_{k+1}$. For correct coordinates the estimate is
the one from Theorem~\ref{th:csc}.\qed
\end{proof}

\subsection*{Proof of Theorem~\ref{th:noise}}
\begin{proof}
In the normalised variables $u_i=(c_i-x^*_i)/w_i$ the step length is
$h_k=(1-\beta)R_k$. Outside the band the sign is correct; inside it
$|u_i-h_ks_i|\le\theta M_k+\eta_k+h_k$. Hence
\[
  M_{k+1}\le\max\{\,M_k-h_k,\ \theta M_k+\eta_k+h_k\,\},
\]
where a possible overshoot past zero is covered by the second term, which is at
least $h_k$. Substituting $M_k\le R_k+d_k$ and using
$\theta+1-\beta=\beta-\Delta$ gives
\[
  M_{k+1}\le R_{k+1}+\max\{\,d_k,\ \theta d_k+\eta_k-\Delta R_k\,\}=R_{k+1}+d_{k+1}.
\]
Put $a_K=\max_{j<K}[\eta_j-\Delta R_j]_+$. If $d\le a_K/(1-\theta)$ then
$\theta d+a_K\le a_K/(1-\theta)$, so the same bound propagates and
$d_K\le a_K/(1-\theta)$, which is~\eqref{eq:noise-finite}. The limit exists
because the step lengths are summable, and letting $K\to\infty$
gives~\eqref{eq:noise-limit}.

For sharpness take $f(x)=x^2/2$, $x^*=0$, $r_0=c_0=p>0$ and the biased oracle
$s(x)=\sgn(x-p)$ with an arbitrary resolution at $x=p$. With
$\eta=(1-\theta)p$ the condition ``outside the band'' reads
$|x|>\theta|x|+(1-\theta)p$, i.e. $|x|>p$, and there the biased oracle is indeed
correct, so the hypotheses hold on the whole line. But these iterates are
exactly sign splitting aimed at the point $p$, hence $c_k\to p$ and
$M_\infty=p=\eta/(1-\theta)$.\qed
\end{proof}

\subsection*{Proof of Proposition~\ref{prop:eta}}
\begin{proof}
Put $v=c_k-x^*$ in physical coordinate units. If
$|v_i|/w_i>\theta\|v\|_{\infty,w}+\eta_k$, the right-hand side of
\eqref{eq:growth} exceeds $a_iw_i\eta_k\ge b_{k,i}$. The estimated derivative
therefore has the same nonzero sign as $v_i$. For a quadratic,
\begin{align*}
 \sgn(v_i)(Hv)_i
 &\ge H_{ii}|v_i|-\sum_{j\ne i}|H_{ij}|w_j\|v\|_{\infty,w}\\
 &\ge H_{ii}w_i\bigl(|v_i|/w_i-\theta_w(H)\|v\|_{\infty,w}\bigr).
\end{align*}
The second relation is an inequality: a particular row's defect may be smaller
than the maximum row defect. At $v_i=0$ its right-hand side is nonpositive,
so the same growth inequality is valid there.\qed
\end{proof}

\subsection*{Finite-difference error}
For a coordinate section $\phi(t)=f(c+te_i)$ whose derivative is
$L_i$-Lipschitz on $[-h,h]$, the fundamental theorem of calculus gives
\[
 \left|\frac{\phi(h)-\phi(-h)}{2h}-\phi'(0)\right|
 \le \frac{1}{2h}\int_{-h}^h L_i|t|\,dt=\frac{L_ih}{2}.
\]
Two value errors of size at most $\nu$ contribute at most $\nu/h$.
Minimising $L_ih/2+\nu/h$ for $h>0$ gives
$h=\sqrt{2\nu/L_i}$ and $\sqrt{2L_i\nu}$ when both parameters are positive.
For $L_i=0$ the bias is zero; a finite available interval limits the radius.
The remaining conclusions follow from Proposition~\ref{prop:eta} and
Theorems~\ref{th:inexact}--\ref{th:noise}.\qed

\section{Modifications and their guarantees}
\label{app:mods}

Throughout this appendix $1/2\le\beta<1$. For each modification we answer two questions: which class it adds and what it
costs in oracle calls. The main-text summary is Table~\ref{tab:mods}.

\subsection{Overlapping splitting: the parameter $\beta$}
\label{sec:beta}

\begin{theorem}[$\beta$-splitting]\label{th:beta}
Let $\beta\in[1/2,1)$, $\theta\in[0,\beta)$, $x^*\in\Bx(c_0,r_0w)=:B_0$ and
$f\in\MSC_w(\theta,x^*,B_0)$. Put
$\lambda=\bigl[\frac{\theta+1-2\beta}{1-\beta}\bigr]_+\in[0,1)$. Then
\begin{equation}\label{eq:beta-bound}
  M_k\ \le\ \beta^kr_0+\lambda\bigl(1-\beta^k\bigr)r_0
\end{equation}
for all $k$. In particular: \textbf{(a)} if $\theta\le2\beta-1$ then
$\lambda=0$ and $M_k\le\beta^kr_0$; \textbf{(b)} with the choice
$\beta=(1+\theta)/2$, accuracy $\eps\in(0,r_0)$ is reached after
$K=\lceil\ln(r_0/\eps)/\ln(1/\beta)\rceil\le
\lceil\frac{2}{1-\theta}\ln(r_0/\eps)\rceil$ gradient evaluations;
\textbf{(c)} for $\beta=1/2$ formula~\eqref{eq:beta-bound} gives
$\lambda=2\theta$, and for that case Theorem~\ref{th:theta} is sharper, with
limiting value $\theta r_0$.
\end{theorem}

\begin{proof}
As in Theorem~\ref{th:theta}, pass to the normalised variables
$u_k=(c_k-x^*)\oslash w$; there the recursion~\eqref{eq:err-rec} reads
$u_{k+1,i}=u_{k,i}-(1-\beta)R_ks_{k,i}$, and condition~\eqref{eq:msc} says that
$|u_{k,i}|>\theta\ninf{u_k}$ forces $s_{k,i}=\sgn(u_{k,i})$, i.e. a step towards
the target.

Put
\[
  \gamma_k:=R_k+\lambda\,(r_0-R_k),\qquad
  \lambda=\lambda(\theta,\beta)=\Bigl[\frac{\theta+1-2\beta}{1-\beta}\Bigr]_+ ,
\]
and prove by induction that $M_k=\ninf{u_k}\le\gamma_k$. The base case is
$M_0\le r_0=\gamma_0$. Assume the bound for $k$ and fix a coordinate $i$.

\emph{Case A: $|u_{k,i}|>\theta M_k$.} The sign is correct, hence
\[
  |u_{k+1,i}|=\bigl||u_{k,i}|-(1-\beta)R_k\bigr|
  \le\max\bigl\{\gamma_k-(1-\beta)R_k,\ (1-\beta)R_k\bigr\}.
\]
The first argument equals $\beta R_k+\lambda(r_0-R_k)\le\beta
R_k+\lambda(r_0-\beta R_k)=\gamma_{k+1}$, since $R_k\ge\beta R_k$ and
$\lambda\ge0$; the second is at most $\beta R_k\le\gamma_{k+1}$ because
$\beta\ge1/2$.

\emph{Case B: $|u_{k,i}|\le\theta M_k\le\theta\gamma_k$.} Then for any sign
(including an arbitrary resolution of a tie)
$|u_{k+1,i}|\le\theta\gamma_k+(1-\beta)R_k$, and the required inequality
$|u_{k+1,i}|\le\gamma_{k+1}$ is equivalent to
\[
  \theta\bigl(R_k+\lambda(r_0-R_k)\bigr)+(1-\beta)R_k
  \ \le\ \beta R_k+\lambda\bigl(r_0-\beta R_k\bigr).
\]
Writing $t=R_k/r_0\in(0,1]$ and dividing by $r_0$ this becomes
\[
  t\bigl[\theta(1-\lambda)+1-2\beta+\lambda\beta\bigr]\ \le\ \lambda(1-\theta).
\]
The right-hand side is nonnegative. If the bracket is non-positive the
inequality holds for all $t\in(0,1]$; otherwise it suffices to check $t=1$,
which gives
\[
  \theta+1-2\beta+\lambda\beta\le\lambda
  \iff
  \lambda\ \ge\ \frac{\theta+1-2\beta}{1-\beta},
\]
and the chosen $\lambda$ satisfies this. The induction is complete, and
\[
  M_k\le\gamma_k=\beta^kr_0+\lambda\bigl(1-\beta^k\bigr)r_0 .
\]
The assumption $\theta<\beta$ gives $\lambda<1$, since $\lambda<1$ is equivalent
to $\theta+1-2\beta<1-\beta$. Part~(a) is the case $\theta\le2\beta-1$, where
$\lambda=0$. Part~(b): for $\beta=(1+\theta)/2$ one has
$\ln(1/\beta)\ge1-\beta=(1-\theta)/2$, whence
$K\le\left\lceil\frac{2}{1-\theta}\ln(r_0/\eps)\right\rceil$. Part~(c): for $\beta=1/2$ the formula
gives $\lambda=2\theta$. \qed
\end{proof}

\keybox{\textbf{The trade-off.} The method with parameter $\beta$ tolerates a
defect $\theta\le2\beta-1$ but contracts at the rate $\beta^k$. The class widens
from $\theta=0$ (at $\beta=1/2$) to $\theta\to1$ (as $\beta\to1$), while the
iteration count grows as $\frac{2}{1-\theta}\ln\frac{r_0}{\eps}$ instead of
$\log_2\frac{r_0}{\eps}$. We stress that the factor $(1-\theta)^{-1}$ is a
property of \emph{this} scheme and not a proved information lower bound for the
class $\MSC(\theta)$; the latter is unknown to us
(Sect.~\ref{sec:discussion}).}

\subsection{Anisotropic boxes: diagonal preconditioning}
\label{sec:box}

By Proposition~\ref{lem:invariance} the choice of the aspect $w$ is exactly a
diagonal preconditioning, and Theorem~\ref{th:perron} identifies the optimal
choice, with minimal defect $\theta_{\mathrm{opt}}=\rho(A)$.
Together with Theorem~\ref{th:quad-iff} this yields the criterion of
solvability of a quadratic problem in its cleanest form
(Corollary~\ref{cor:hmat}). The accompanying costs---an enlarged radius and the
knowledge of $H$---are discussed in Remark~\ref{rem:cost-w}.

\subsection{Freezing coordinates}
\label{sec:freeze}

Lemma~\ref{lem:lock} says that what is fatal is not the sign error itself but
the \emph{obligation} to step. If no step is made along an ``uncertain''
coordinate and the corresponding half-size is not reduced, the error is not
locked.

\begin{definition}[\CUBEF\ and the class $\DSC_\epsilon$]\label{def:freeze}
Let $0<\epsilon\le1$. The method \CUBEF$(\beta,w,\epsilon)$ differs from
Algorithm~\ref{alg:cube} in that at step $k$ the set of active coordinates
\[
  S_k=\bigl\{i:\ |\partial_if(c_k)|\ge\epsilon\,\ninf{\nabla f(c_k)}\bigr\}
\]
is formed (with $S_k=[n]$ if $\nabla f(c_k)=0$), and the updates
$c_{k+1,i}=c_{k,i}-(1-\beta)r_{k,i}s_{k,i}$, $r_{k+1,i}=\beta r_{k,i}$ are
performed only for $i\in S_k$; for $i\notin S_k$ both $c_{k,i}$ and $r_{k,i}$
are left unchanged. The class $\DSC_\epsilon(x^*,B)$ consists of the functions
for which
\begin{equation}\label{eq:dsc}
 \begin{gathered}
 |\partial_if(x)|\ge\epsilon\|\nabla f(x)\|_\infty,\qquad x_i\ne x_i^*\\
 \Longrightarrow\quad \partial_if(x)(x_i-x_i^*)>0,
 \qquad x\in B,\quad i\in[n].
 \end{gathered}
\end{equation}
\end{definition}

The proviso $x_i\ne x^*_i$ in~\eqref{eq:dsc} is essential: without it the
condition would fail at the target, where the required strict product is zero. With the proviso everything is consistent and, in
particular, $\CSC(x^*,B)\subseteq\DSC_\epsilon(x^*,B)$ for every
$\epsilon\in(0,1]$.

\begin{theorem}[freezing]\label{th:freeze}
Let $x^*\in\Bx(c_0,r_0w)$ and let
$f\in\DSC_\epsilon\bigl(x^*,\Bx(c_0,r_0w)\bigr)$. Then
$|c_{k,i}-x^*_i|\le r_{k,i}$ for all $k$ and $i$, where
$r_{k,i}=\beta^{m_i(k)}r_0w_i$ and $m_i(k)=\#\{j<k:\ i\in S_j\}$ is the number
of activations of coordinate $i$. Consequently
$M_k\le\beta^{\,\min_im_i(k)}r_0$, and if every coordinate is activated in at
least a fraction $\nu$ of the steps, then $M_k\le\beta^{\nu k}r_0$.
\end{theorem}

\begin{proof}
Induction on $k$ with the invariant $|e_{k,i}|\le r_{k,i}$. For $i\notin S_k$
neither $c_{k,i}$ nor $r_{k,i}$ changes. For $i\in S_k$ with $e_{k,i}=0$ we get
$|e_{k+1,i}|=(1-\beta)r_{k,i}\le\beta r_{k,i}$; for $e_{k,i}\ne0$
condition~\eqref{eq:dsc} gives a correct non-degenerate sign and the computation
of Theorem~\ref{th:csc} applies.\qed
\end{proof}

Note that the guarantee is stated through the number of activations and does not
turn into unconditional linear convergence without a further assumption: a
permanently frozen coordinate is never refined. Numerical improvement from freezing does not establish membership in $\DSC_\epsilon$; these are different statements.

\subsection{Smoothing: global optimisation of multimodal functions}
\label{sec:smooth}

\begin{definition}[\CUBES]\label{def:smooth}
The method \CUBES$(\beta,w,\tau,m)$ uses, instead of $\nabla f(c_k)$, the
estimate $\hat g_k=\frac1m\sum_{t=1}^m\nabla f(c_k+\tau u_t)$ with independent
$u_t\sim\mathcal{U}[-1,1]^n$ and a \emph{fixed} radius $\tau$ (independent of
the iteration). Here $f$ is $C^1$ on a neighbourhood of all probe boxes, so
its locally bounded continuous gradient may be averaged under the integral.
For $m=\infty$ (an exact smoothed oracle)
$\hat g_k=\nabla f_\tau(c_k)$, where
$f_\tau(x)=\E_{u}f(x+\tau u)$.
\end{definition}

\begin{theorem}[reduction to the smoothed problem]\label{th:smooth}
Let $f_\tau\in\CSC(x^*_\tau,B)$, where $x^*_\tau$ is the minimiser of $f_\tau$,
and let $x^*_\tau\in\Bx(c_0,r_0w)\subseteq B$. Then \CUBES\ with an exact
smoothed oracle is exactly \CUBE\ applied to $f_\tau$, hence
$\ninfw{c_k-x^*_\tau}\le\beta^kr_0$; the error relative to the original target
does not exceed $\beta^kr_0\|w\|_\infty+\ninf{x^*_\tau-x^*}$.\qed
\end{theorem}

The content of Theorem~\ref{th:smooth} depends on how much smoothing widens the
class. The next proposition is a completely explicit example in which smoothing
turns a function with exponentially many local minima into a function of class
$\CSC$.

\begin{proposition}[a model multimodal family]\label{prop:rastrigin}
Let
\[
  f(x)=\|x-x^*\|_2^2+A\sum_{i=1}^n\bigl(1-\cos2\pi(x_i-x^*_i)\bigr),
  \qquad A>0 .
\]
Then $\partial_if_\tau(x)=2(x_i-x^*_i)+c(\tau)\sin2\pi(x_i-x^*_i)$ with
\[
  c(\tau)=2\pi A\,s(\tau),\qquad s(\tau)=\frac{\sin2\pi\tau}{2\pi\tau},
\]
where $s(0)=1$, and $f_\tau\in\CSC(x^*,\Rn)$ if and only if
\begin{equation}\label{eq:smooth-window}
  -\frac1\pi\ \le\ c(\tau)\ <\ c_+,
  \qquad
  c_+=\min_{t\in(1/2,1)}\frac{2t}{-\sin2\pi t}=1.46528\ldots
\end{equation}
In particular: \textbf{(i)} for $\tau\ge\pi A$ the condition always holds, since
$|c(\tau)|\le A/\tau\le1/\pi$; \textbf{(ii)} for $\tau=k/2$, $k\in\mathbb{N}$,
with $k\ge1$, one has $s(\tau)=0$ and $f_\tau(x)=\|x-x^*\|^2+\mathrm{const}$: a ``resonant''
radius annihilates the harmonic exactly, whatever the amplitude $A$.
\end{proposition}

\begin{proof}
The formula for $\partial_if_\tau$ follows by averaging over independent
coordinates. Condition~\eqref{eq:csc} for $f_\tau$ is equivalent to
$h(t):=2t+c\sin2\pi t>0$ for all $t>0$. For $c\ge0$ a violation is possible only
where $\sin2\pi t<0$, i.e. $t\in(m+\frac12,m+1)$; the inequality
$c<2t/(-\sin2\pi t)$ on all such intervals is equivalent to $c<c_+$, the minimum
being attained on the first one. For $c<0$, positivity near $0$ requires $c\ge-1/\pi$.
This condition is also sufficient: $\sin z<z$ for $z>0$ gives
$2t+c\sin(2\pi t)>0$ whenever the sine is positive and $|c|\le1/\pi$;
when the sine is nonpositive the conclusion is immediate. In particular the
boundary is included, since $h(t)=(2\pi t-\sin2\pi t)/\pi>0$ for $t>0$.
The upper boundary stays strict,
since at $c=c_+$ the function $h$ acquires a zero.\qed
\end{proof}

\begin{remark}[how many minima, exactly]\label{rem:minima}
Let $q(A)$ count the roots of the section derivative, and let $m(A)$ be the
number of \emph{local minima} of a one-dimensional section. All roots lie in
$|t|\le\pi A$, so for a large enough $R$ the function $f$ has exactly $q(A)^n$
critical points in a cube of radius $R$ and exactly $m(A)^n$ local minima. The
identity $m(A)=(q(A)+1)/2$ holds only when all roots are simple; at exceptional
amplitudes, described by $F'(t)=F''(t)=0$ for $F(t)=t^2+A(1-\cos2\pi t)$, a
double root appears without a sign change of the derivative and the formula
fails. The first such amplitude is explicit: the double-root condition is
equivalent to $\tan z=z$ with $z=2\pi t$, whence $z_0\approx4.493409$,
$t_0=z_0/(2\pi)\approx0.715148$ and
$A_{\mathrm{crit}}=c_+/(2\pi)\approx0.233208$. At $A=A_{\mathrm{crit}}$ there
are three critical points ($-t_0$, $0$, $t_0$), so $q=3$, but only one local
minimum, whereas $(q+1)/2$ would give two. Note also that
$A_{\mathrm{crit}}$ is the threshold of multimodality, while the threshold of
convexity is $1/(2\pi^2)\approx0.0507$ (from
$F''(t)=2+4\pi^2A\cos2\pi t\ge2-4\pi^2A$): for
$1/(2\pi^2)<A<A_{\mathrm{crit}}$ the function is already non-convex but still
has a unique minimiser.
\end{remark}

\begin{remark}[which oracle is being paid for]\label{rem:mc}
Theorem~\ref{th:smooth} is proved for the \emph{exact} oracle $\nabla f_\tau$.
For finite $m$ the estimate is noisy, and the signs may be unreliable when
$|\partial_if_\tau(c_k)|$ drops to the level $\sigma_{\mathrm{MC}}/\sqrt m$,
where $\sigma_{\mathrm{MC}}$ is the standard deviation of one summand. For the
family of Proposition~\ref{prop:rastrigin} the gradient component at a random
point equals $2(x_i-x^*_i)+2\tau u_i+2\pi A\sin(\cdot)$, whence
\[
  \sigma_{\mathrm{MC}}\ \le\ \frac{2\tau}{\sqrt3}+2\pi A
\]
(the first term is the contribution of the linear part, with
$\mathrm{Var}(2\tau u_i)=4\tau^2/3$); for large $\tau$ it dominates. The product
$mK$ is therefore the cost of a particular run, \emph{not} a proved complexity
of reaching $\eps$, as long as $m$ is not tied to $\eps$ and $\delta$; a quantitative guarantee additionally needs a concentration bound and a
growth condition converting gradient error to a displacement band, as in
Proposition~\ref{prop:eta} and Theorem~\ref{th:inexact}. A standard-deviation
bound alone does not yield a simultaneous sign guarantee. Accordingly, in Table~\ref{tab:mods} the two costs are
listed separately, and we do not claim that smoothing is superior to multistart
in general: only regimes with an equally defined oracle cost are comparable.
\end{remark}

\begin{remark}[two queries instead of averaging]\label{rem:twoquery}
For the same family the harmonic can be annihilated \emph{exactly} by two
ordinary gradient queries. With $v=\tfrac14\mathbf1$,
\begin{equation}\label{eq:twoquery}
  \tfrac12\bigl(\nabla f(x+v)+\nabla f(x-v)\bigr)=2(x-x^*),
\end{equation}
because the sines with arguments $2\pi(x_i-x^*_i)\pm\pi/2$ cancel. This is not
the same as uniform smoothing (only the signs of the resulting field agree) but
a special exact sign oracle for a specific family, and it requires the ability
to query the gradient at points that may leave the initial box. Given the \emph{values} of both gradients and the normalisation of the
quadratic part, \eqref{eq:twoquery} recovers the solution immediately,
$x^*=x-\bar g(x)/2$. Hence~\eqref{eq:twoquery} is useful as a control example on
the cost of an oracle, not as an argument for the superiority of box splitting
in global optimisation.
\end{remark}

\subsection{A noisy sign oracle and majority voting}
\label{sec:vote}

\begin{definition}[noise model]\label{def:noise}
On a query at $x$ the oracle returns $\hat s\in\{-1,1\}^n$ such that for every
coordinate with $\partial_if(x)\ne0$
\[
  \Prob\bigl[\hat s_i=\sgn\partial_if(x)\,\big|\,\mathcal{H}\bigr]
  \ \ge\ \tfrac12+\gamma
\]
conditionally on any history $\mathcal{H}$; the answers are conditionally
independent across queries, and for the bound~\eqref{eq:vote-max} independence
across coordinates is not needed. In coordinates with $\partial_if(x)=0$ the
answer is arbitrary and harmless. The method \CUBEV\ makes $m$ (odd) queries and
takes the coordinatewise majority. We stress that a constant $\gamma$ is a
\emph{modelling assumption}: it does not follow automatically from additive
noise in the values or in the gradient, because the gap between the compared
quantities vanishes near the solution (Sect.~\ref{sec:inexact}).
\end{definition}

\begin{theorem}[voting]\label{th:vote}
Let $f\in\CSC(x^*,B_0)$ with $x^*\in B_0=\Bx(c_0,r_0w)$, and let $p_m$ be an
\emph{upper bound}, uniform over points and histories, on the conditional
probability that the majority errs in a given coordinate with
$\partial_if\ne0$; by Hoeffding's inequality $p_m=e^{-2\gamma^2m}$ is
admissible. Then:
\begin{enumerate}[leftmargin=*,label=\textup{(\roman*)},itemsep=1pt]
\item \textup{(coordinatewise)}
  $\displaystyle\E\lim_{k\to\infty}\frac{|c_{k,i}-x^*_i|}{w_i}\le2p_mr_0$;
\item \textup{(in the maximum norm)}
  \begin{equation}\label{eq:vote-max}
    \E\lim_{k\to\infty}M_k\ \le\ 2r_0\min\{1,\,np_m\},
  \end{equation}
  and, if the answers are conditionally independent across coordinates, also
  $\E\lim_kM_k\le2r_0\bigl[1-(1-p_m)^n\bigr]$; for a finite $K$ the remainder
  $\beta^Kr_0$ is to be added;
\item \textup{(with high probability)} for any $K$ and $\delta\in(0,1)$, if
  $m\ge\frac{1}{2\gamma^2}\ln\frac{nK}{\delta}$, then with probability at least
  $1-\delta$ all signs on the first $K$ steps are correct and
  $M_K\le\beta^Kr_0$.
\end{enumerate}
\end{theorem}

\begin{proof}
(i) By Lemma~\ref{lem:master} the final error in coordinate $i$ is at most
\[
 \sum_{k\ge0}2(1-\beta)r_{k,i}\,
 \mathbf1\{\text{step }k\text{ is wrong in coordinate }i\}.
\]
Since
$f\in\CSC$, a step can be wrong only through an oracle error, whose conditional
probability does not exceed $p_m$; taking expectations gives
$2p_mr_0w_i$.

(ii) Claim (i) cannot be carried over to the maximum, because
$\E\max_iZ_i\le\max_i\E Z_i$ is false. Instead,
\[
  \lim_kM_k\ \le\ \sum_{k\ge0}2(1-\beta)R_k\,
    \mathbf1\{\text{at step }k\text{ at least one coordinate is wrong}\},
\]
and the probability of the event is bounded by a union bound ($\le np_m$; no
independence needed) or, under conditional independence, by $1-(1-p_m)^n$.
Multiplying by $\sum_k2(1-\beta)\beta^kr_0=2r_0$ gives both bounds. Note that
$1-(1-p_m)^n$ is an upper bound and not an exact value: equality would require
equal error probabilities and all $n$ coordinates to be relevant
($e_{k,i}\ne0$).

(iii) A union bound over $K$ steps and $n$ coordinates gives failure probability
at most $nKp_m\le nKe^{-2\gamma^2m}\le\delta$; on the complement
Theorem~\ref{th:csc} applies.\qed
\end{proof}

\begin{example}[the dimension must not be lost]\label{ex:vote-dim}
Let $n=3$, $f(x)=\|x\|_2^2/2$, $x^*=0$, $w=\mathbf1$, $c_0=\mathbf1$, $r_0=1$,
$\beta=1/2$, $m=1$, and let each sign be independently wrong with probability
$p=0.1$. If at least one sign is wrong at the first step, then
$\lim_kM_k\ge1$ by Lemma~\ref{lem:lock}. Hence
\[
  \E\lim_kM_k\ \ge\ 1-(1-p)^3=0.271\ >\ 0.2=2pr_0 ,
\]
so the coordinatewise bound (i) is invalid for the maximum norm. The bound
$2r_0[1-(1-p)^3]=0.542$ does hold; a reproducible finite-budget experiment is given in Appendix~\ref{app:experiments}.
\end{example}

\subsection{Multistart}
\label{sec:multistart}

\begin{theorem}[multistart]\label{th:multistart}
Let $G\subseteq Q$ be the set of centres from which the method reaches an
$\eps$-neighbourhood of the global minimiser within a fixed budget, and let
$\pi=\mathrm{vol}(G)/\mathrm{vol}(Q)>0$. Then among $M$ independent uniform
starts at least one is successful with probability
$1-(1-\pi)^M\ge1-e^{-\pi M}$; for reliability $1-\delta$ it suffices to take
$M\ge\pi^{-1}\ln(1/\delta)$.\qed
\end{theorem}

\begin{remark}[what is guaranteed, and how to select]
The theorem guarantees the \emph{existence} of a successful run but gives no
rule for identifying it. If the goal is $\eps$-optimality \emph{in function
value}, then $M-1$ exact comparisons of the final candidates select a point no
worse than the successful one, and the guarantee is preserved. For
$\eps$-closeness \emph{in the argument} this is not enough: an additional growth
condition, e.g. $f(x)-f^*\ge\mu\,\mathrm{dist}(x,x^*)^\alpha$, is required.
Finally, if $G$ is a basin of size $\rho$ inside a cube of size $R$, then
$\pi\approx(\rho/R)^n$ and $M\approx(R/\rho)^n\ln(1/\delta)$: multistart widens
the class but pays an exponential price.
\end{remark}

\section{The probabilistic setting}
\label{app:random}

Worst-case and ensemble statements answer different questions. The setting ``the instance is drawn at
random from a class; the method reaches the required accuracy with probability
at least $1-\delta$'' requires an ensemble to be fixed. A tractable model is a Gaussian perturbation of the identity.

\begin{definition}[the ensemble $\mathcal{E}_n(\sigma)$]\label{def:ensemble}
Let $W$ be symmetric with independent (up to symmetry) entries
$W_{ij}\sim\mathcal{N}(0,1)$ for $i\ne j$ and $W_{ii}\sim\mathcal{N}(0,2)$; this
is the standard Gaussian orthogonal ensemble (GOE). The ensemble
$\mathcal{E}_n(\sigma)$ consists of the quadratic functions
$f(x)=\frac12\ip{H(x-x^*)}{x-x^*}$ with
$H=I+\frac{\sigma}{\sqrt n}\,W$, $\sigma\ge0$.
\end{definition}

Positive definiteness and a universal sign certificate are different events.
The following finite-dimensional bounds avoid relying on a spectral-limit
approximation.

\begin{theorem}[explicit bounds on the defect]\label{th:ensemble}
Let $n\ge2$, $0<\delta<1$ and
\[
  \ell=\ln\frac{4n}{\delta},\quad
  a=\sqrt{2(n-1)\ell},\quad
  b=2\sigma\sqrt{\frac{\ell}{n}},\quad
  m_0=\sqrt{2/\pi}\approx0.7979 .
\]
If $b<1$, then with probability at least $1-\delta$ one has $H_{ii}>0$ for all
$i$ and
\begin{equation}\label{eq:GOE-explicit}
  \frac{\sigma}{\sqrt n}\cdot\frac{\bigl[(n-1)m_0-a\bigr]_+}{1+b}
  \ \le\ \theta_{\mathrm{opt}}(H)\ \le\
  \frac{\sigma}{\sqrt n}\cdot\frac{(n-1)m_0+a}{1-b}.
\end{equation}
\end{theorem}

\begin{proof}
For $\sigma=0$ the assertion is deterministic; assume $\sigma>0$.
Write $S_i=\sum_{j\ne i}|W_{ij}|$. This is a function of $n-1$ independent
standard Gaussians, Lipschitz with constant $\sqrt{n-1}$ in the Euclidean norm,
so by the Gaussian concentration inequality~\cite[Theorem 3.25]{VanHandel}
\[
  \Prob\bigl\{|S_i-\E S_i|>a\bigr\}\le2e^{-a^2/(2(n-1))}=\frac{\delta}{2n},
  \qquad \E S_i=(n-1)m_0 .
\]
For the diagonal, $W_{ii}\sim\mathcal N(0,2)$ and
$\Prob\{|H_{ii}-1|>b\}\le2e^{-nb^2/(4\sigma^2)}=\frac{\delta}{2n}$. A union
bound over $i$ produces an event of probability at least $1-\delta$ on which
simultaneously $1-b\le H_{ii}\le1+b$ (in particular $H_{ii}>0$, since $b<1$) and
$|S_i-(n-1)m_0|\le a$ for all $i$.

On that event consider $A=D^{-1}|H-D|\ge0$, whose row sums are
$\frac{\sigma}{\sqrt n}S_i/H_{ii}$. For a nonnegative matrix the spectral radius
is squeezed between the smallest and the largest row sum~\cite[Ch.
2]{BermanPlemmons1994}, and $\theta_{\mathrm{opt}}(H)=\rho(A)$ by
Theorem~\ref{th:perron}. Substituting the bounds for $S_i$ and $H_{ii}$
yields~\eqref{eq:GOE-explicit}.\qed
\end{proof}

\begin{corollary}[the criterion and its quantifiers]\label{cor:ensemble}
On the event of Theorem~\ref{th:ensemble}:
\begin{enumerate}[leftmargin=*,label=\textup{(\alph*)},itemsep=1pt]
\item if the right-hand side of~\eqref{eq:GOE-explicit} is less than $1$, then
  $H$ is positive definite and strictly generalised diagonally dominant and by
  Corollary~\ref{cor:hmat} there \emph{exists} a pair $(\beta,w^*)$ for which
  the method converges linearly from \emph{all} boxes containing the solution;
\item if the left-hand side exceeds $1$, then $\theta_{\mathrm{opt}}(H)>1$, and
  by Theorem~\ref{th:quad-iff} \emph{for every} fixed pair $(\beta,w)$ there
  \emph{exists} a starting box on which the method fails.
\end{enumerate}
Note the order of the quantifiers: crossing the threshold changes the
possibility of a \emph{universal guarantee over all starting boxes}, not the
probability of failure from a random start.
\end{corollary}

\begin{remark}[asymptotics versus a certificate]\label{rem:asympt}
For $0<\sigma=O(n^{-1/2})$ and $\ln(n/\delta)=o(n)$ both sides
of~\eqref{eq:GOE-explicit} are asymptotically equal to
$\tau:=\sigma\sqrt{2n/\pi}$, so the threshold picture ``$\tau=1$'' survives, now
with an explicit parameter range and an explicit event. It is useful to keep the
difference between an asymptotic and a certificate in mind: for $n=50$,
$\delta=0.01$ and $\sigma=0.15$ the right-hand side of~\eqref{eq:GOE-explicit}
is about $1.72$, and a rigorous certificate of strict diagonal dominance is
obtained only for $\sigma<0.0924$. This does not mean that the method fails at
$\sigma=0.15$---it means that a conservative concentration bound is not yet a proof at that
dimension.
\end{remark}

The event $\rho(A)<1$ certifies a suitable geometry for every admissible
starting box. A finite-budget success frequency from random starts is a
different quantity. Figure~\ref{fig:random} reports both with their denominators
and uncertainty intervals; neither is inferred from the other.

\subsection*{Sparse ensembles}

\begin{proposition}[deterministic bound]\label{prop:sparse}
Let $H$ be symmetric with $H_{ii}>0$ and at most $d\ge1$ nonzero off-diagonal
entries in each row. Then
\[
  \theta_{\mathrm{opt}}(H)\ \le\ \theta_{\mathbf1}(H)\ \le\
  d\,\frac{\max_{i\ne j}|H_{ij}|}{\min_iH_{ii}} . \eqno\qed
\]
\end{proposition}

\begin{proposition}[random $d$-regular ensemble]\label{prop:sparse-rand}
Let $G$ be a fixed $d$-regular graph on $n$ vertices, with $d\ge1$,
$\sigma\ge0$, $0<\delta<1$ and $H_{ii}=1$. Let the
edge weights $H_{ij}=H_{ji}=\sigma Z_{ij}/\sqrt d$ be independent with
$Z_{ij}\sim\mathcal N(0,1)$. Then with $\ell=\ln(2n/\delta)$ and
$a=\sqrt{2d\ell}$, with probability at least $1-\delta$,
\[
  \frac{\sigma}{\sqrt d}\bigl[d\,m_0-a\bigr]_+\ \le\
  \theta_{\mathrm{opt}}(H)\ \le\
  \frac{\sigma}{\sqrt d}\bigl(d\,m_0+a\bigr).
\]
Both sides are relatively close to $\sigma\sqrt{2d/\pi}$ only when
$\ell=o(d)$: the weaker condition $d\gtrsim\ell$ gives control of the scale with
a constant relative error only. At fixed $d$ and unbounded Gaussian weights,
rare large edges become significant as $n\to\infty$, so the statement ``the
threshold $\sigma_{\mathrm{crit}}\approx\sqrt{\pi/(2d)}$ does not depend on the
dimension'' is to be read as a statement about the leading scale under
$\ln(n/\delta)=o(d)$, not as a guarantee uniform in $n$.
\end{proposition}

\begin{proof}
A verbatim repetition of the proof of Theorem~\ref{th:ensemble} with $n-1$
replaced by $d$ and a deterministic diagonal.\qed
\end{proof}

\begin{remark}[why regularity is needed]\label{rem:star}
The condition ``at most $d$ entries per row'' does not suffice for a two-sided
bound. Consider a star: vertex $0$ joined to $d$ leaves, $H_{00}=H_{jj}=1$,
$H_{0j}=\sigma Z_j/\sqrt d$. Here $D^{-1}|H-D|$ has nonzero entries only in the
first row and column, and its spectral radius equals
$\frac{\sigma}{\sqrt d}\bigl(\sum_{j=1}^dZ_j^2\bigr)^{1/2}\to\sigma$ as
$d\to\infty$---of order $\sigma$, not $\sigma\sqrt d$. (Directly: a matrix of
the form $\bigl(\begin{smallmatrix}0&v^{\!\top}\\v&0\end{smallmatrix}\bigr)$ has
spectral radius $\|v\|_2$.) For a fixed $\sigma<1$ the corresponding $H$ is
positive definite with probability tending to one, so the example stays inside
convex optimisation.
\end{remark}

\section{Additional numerical evidence}\label{app:experiments}

This appendix supplements the matched comparisons in Sect.~\ref{sec:experiments}.
The examples distinguish a proved invariant, a finite-instance observation,
and a probability estimated from independent repetitions. No experiment replaces
a universal quantifier in a theorem. All displayed figures are generated from
saved numerical data; captions identify the oracle, budget and error metric.

\subsection{An exclusion that no later step can repair}
Figure~\ref{fig:counterex} shows the strongly convex quadratic in
Counterexample~\ref{cex:2x2}. Its first gradient is $(-0.1,-1.1)$, so both
coordinates move upwards: $c_1=(1.4,-0.5)$. The first child already excludes
$0$. The narrow sequence of later boxes cannot cross back over its left
boundary $x_1=0.9$. Smooth elliptical contours therefore do not protect this
particular greedy localisation scheme from an irreversible mistake.

\subsection{Random Hessians: universal certificates and random-start outcomes}
For $n=30$ and $15$ equally spaced values $\tau\in[0.1,2.2]$, generate
$60$ independent matrices per value according to
\[
 \sigma=\frac{\tau}{\sqrt{2n/\pi}},\qquad H=I+\frac{\sigma}{\sqrt n}W,
\]
with the GOE convention of Sect.~\ref{sec:extensions}. Targets and initial
centres are independently uniform on $[-1,1]^n$. All frequencies have
denominator $60$. Non-positive-definite matrices are recorded, not resampled;
for performance curves they count as failures. This makes the displayed
frequencies unconditional over the specified ensemble.

On positive definite instances the two methods receive $260$ gradients each.
The first uses $w=\mathbf1$, $\beta=1/2$. The second uses Perron aspect $w$
and $\beta=\min\{0.95,(1+\rho(A))/2\}$. The latter is a \emph{fixed-aspect}
rule, not the adaptive half-size algorithm of Theorem~\ref{th:adaptive}.
Each initial box contains its target. Success is measured using the same
unweighted criterion
$\|c_{260}-x^*\|_\infty/\|c_0-x^*\|_\infty\le10^{-2}$, rather than a different
weighted norm for each method. Matrix setup is separate from the gradient
budget. If $\rho(A)>0.9$, the cap $\beta=0.95$ may sacrifice the universal
certificate even when $\rho(A)<1$.

\begin{figure}[htbp]
\centering\includegraphics[width=\textwidth]{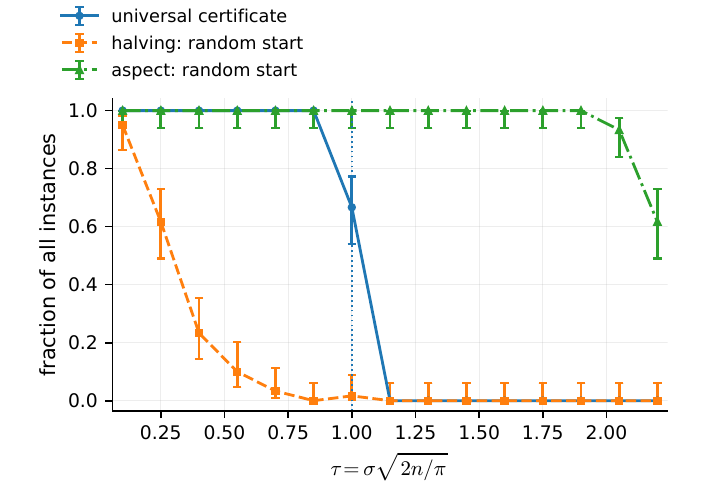}
\caption{GOE experiment, $n=30$, $60$ independent instances per horizontal
value. Error bars are nominal $95\%$ Wilson intervals for unconditional
binomial frequencies. The certificate event is $\rho(A)<1$; the other curves
measure success from sampled starts after $260$ gradients. The line $\tau=1$
marks an asymptotic scale, not an exact finite-dimensional probability threshold.
Failures of a universal certificate and successes of individual runs can coexist}
\label{fig:random}
\end{figure}

\subsection{The cost of an anisotropic box}
For the one-dimensional Dirichlet Laplacian of size $n$,
$H=\operatorname{tridiag}(-1,2,-1)$, the analytic Perron weights are
$w_i=\sin(i\pi/(n+1))$ and $\theta=\cos(\pi/(n+1))$.
We normalise $\max_iw_i=1$ and set $\beta=(1+\theta)/2$. The computed
weights are checked against this analytic vector. Box splitting and Jacobi
start from the \emph{same} point uniform on $[-1,1]^n$ and stop at relative
weighted error $10^{-6}$. Jacobi uses $x_{k+1}=x_k-D^{-1}Hx_k$ with the exact
zero target; a matrix--vector product is charged at each step.

Figure~\ref{fig:laplace} and Table~\ref{tab:laplace} compare these iteration counts.
\begin{figure}[htbp]
\centering\includegraphics[width=\textwidth]{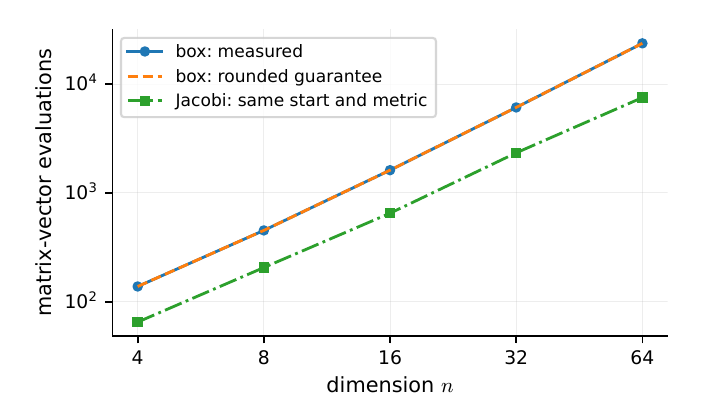}
\caption{Laplacian, common initial point and weighted stopping criterion.
Measured box iterations coincide here with the rounded worst-case guarantee;
this equality is an observation for these starts, not an assertion for all
initial points. Jacobi is a full-magnitude method and its spectral cancellation
is not captured by an absolute sign-coupling matrix}
\label{fig:laplace}
\end{figure}

\begin{table}[htbp]
\caption{Laplacian counts at relative weighted tolerance $10^{-6}$. The last
column is the unrounded aspect overhead $\log(\max w/\min w)/\log(1/\beta)$,
not part of the stopping count for an already admissible weighted box.}
\label{tab:laplace}
\begin{tabularx}{\textwidth}{@{}rLrrr@{}}
\toprule
$n$ & $\theta$ & Box & Jacobi & Aspect overhead\\
\midrule
4 & 0.809017 & 138 & 65 & 4.79\\
8 & 0.939693 & 452 & 206 & 34.54\\
16 & 0.982973 & 1616 & 650 & 197.65\\
32 & 0.995472 & 6096 & 2332 & 1037.74\\
64 & 0.998832 & 23655 & 7509 & 5187.43\\
\bottomrule\end{tabularx}\end{table}

The extra iterations required to enclose a unit unweighted cube with this
aspect are governed by
$\log(\max w/\min w)/\log(1/\beta)=\Theta(n^2\log n)$.
Thus optimising the contraction alone can conceal a growing initial-radius
cost. Conversely, for $\sum_i|x_i-x_i^*|^{1.6}$ with $r_0=1$ and targets
uniform on $[-0.7,0.7]^n$, dimensions
$10,100,1000,10^4,10^5$ all reached absolute error $10^{-8}$ in $27$ halvings.
The guarantee is dimension-independent in \emph{vector queries}, while data
storage and coordinate updates remain at least linear in $n$.

\subsection{Persistent error: a band and a capped fixed point}
For the abstract sign-band experiment, take $\beta=0.7$, $\theta=0.2$,
$\eta_k=0.05$, $r_0=1$ in dimension $4$. Use $400$ independent initial
vectors, normalised to maximum norm $1$, and $60$ steps. A sign is chosen
wrong whenever the corresponding coordinate is inside the permitted band,
and correctly outside it. This is a stress test of the \emph{sign assumption};
it does not assert that every such sequence is the gradient of a smooth
potential. Figure~\ref{fig:band} plots the largest observed error together
with the deterministic bound of Theorem~\ref{th:noise}.

\begin{figure}[htbp]
\centering\includegraphics[width=\textwidth]{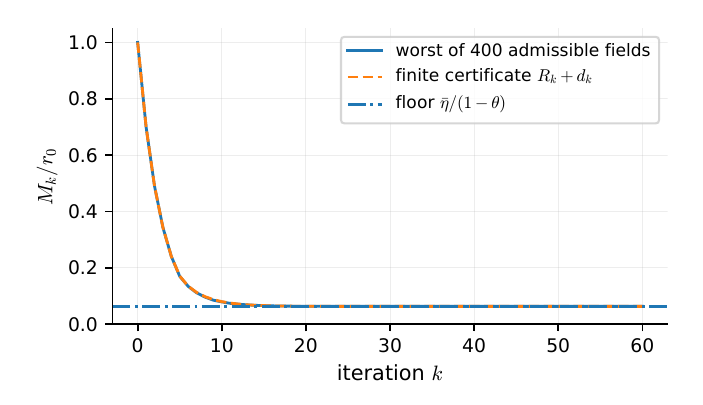}
\caption{Admissible adversarial sign sequences, not independently generated
objective functions. The curve $R_k+d_k$ is the proved finite-time upper bound.
The persistent-band limit $\bar\eta/(1-\theta)=0.0625$ is an upper bound for
this experiment; its sharpness is proved separately by a one-dimensional witness}
\label{fig:band}
\end{figure}

For capped adaptive half-sizes, use the nonnegative matrix and persistent
coordinate error
\[
 A=\begin{pmatrix}0&0.05\\5&0\end{pmatrix},\quad
 \xi=\binom{0.03}{0.04},\quad r^{\mathrm{init}}=\binom{0.02}{1}.
\]
Here $\rho(A)=1/2$, and $A=D^{-1}|H-D|$ for the positive definite
$H=\left(\begin{smallmatrix}100&5\\5&1\end{smallmatrix}\right)$.
The unrestricted resolvent is $g=(0.042\overline6,0.253\overline3)$, whereas
$v=(0.02,0.14)$ solves $v=\min\{r^{\mathrm{init}},Av+\xi\}$.
The smaller first half-size is never needlessly enlarged.

Figure~\ref{fig:capped} shows the coordinatewise radius limit.
\begin{figure}[htbp]
\centering\includegraphics[width=\textwidth]{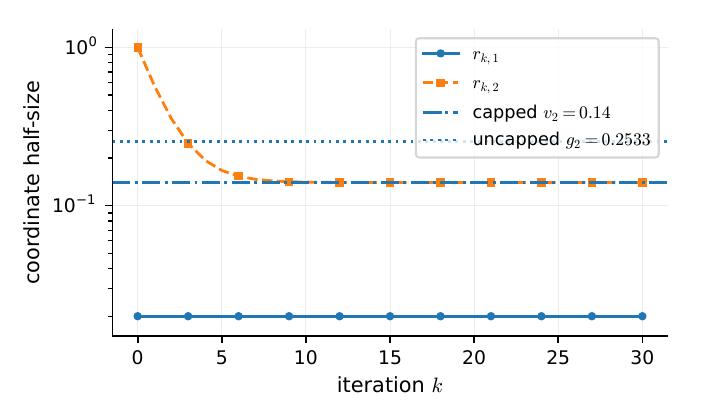}
\caption{Adaptive half-sizes with persistent error. The curves are radii, not
centre errors. The first coordinate remains capped at $0.02$ and the second
converges to $0.14$, strictly below the second component of $(I-A)^{-1}\xi$.
Corollary~\ref{cor:capped-limit} explains both the limit and the finite-time rate}
\label{fig:capped}
\end{figure}

\subsection{Exact and sampled smoothing are different experiments}
Figure~\ref{fig:smooth} illustrates Proposition~\ref{prop:rastrigin} at $A=1$
and fixed $\tau=1.5$. Exact uniform averaging removes the oscillation in every
coordinate derivative. It does not imply that a finite Monte Carlo average
has exact signs.

\begin{figure}[htbp]
\centering\includegraphics[width=\textwidth]{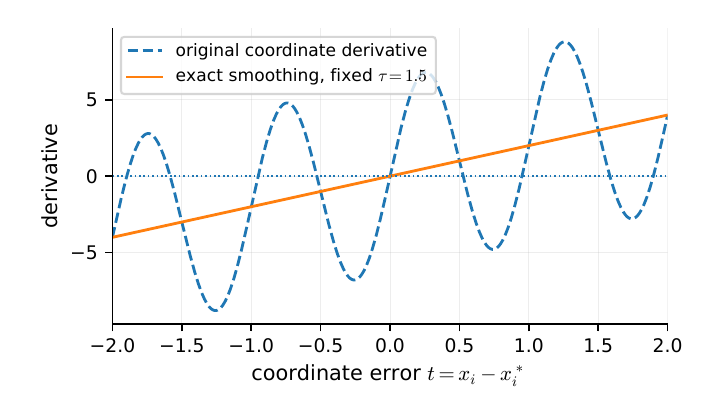}
\caption{Coordinate derivatives before and after \emph{exact} smoothing of
the multimodal quadratic--cosine family. The radius stays fixed at $\tau=1.5$;
the smoothed derivative is $2t$. No sampling is used in this plot}
\label{fig:smooth}
\end{figure}

For the sampled experiment, use $n\in\{3,10\}$, $200$ starting points per
sample size, $K=25$, $r_0=1$, $x^*=0$, $A=1$ and $\tau=1.5$.
Starting centres are uniform on $[-1,1]^n$ and shared across sample sizes.
At every step, average $m\in\{2,8,32,128,512\}$ independently sampled
\emph{original} gradients at $c_k+\tau u$, then halve the box according to
the averaged signs. Each run therefore costs $mK$ raw gradients. The same
fixed smoothing radius is used throughout the run; only the box shrinks.

Figure~\ref{fig:mc} separates the sampled-oracle error from the exact-smoothed bound.
\begin{figure}[htbp]
\centering\includegraphics[width=\textwidth]{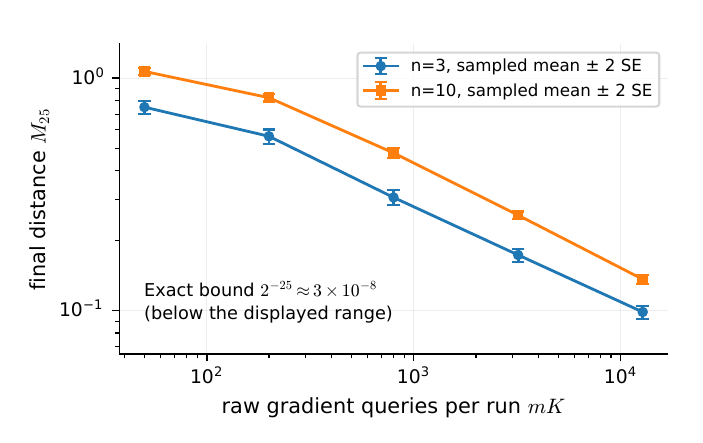}
\caption{Sampled smoothing at a fixed radius. Points are mean maximum-norm
errors over $200$ runs; bars show two estimated standard errors, not uniform
high-probability bounds. The annotated $2^{-25}$ guarantee for an \emph{exact} smoothed oracle lies
below the displayed range. Increasing the sampling budget reduces the
observed error but does not make the exact-oracle theorem automatically
applicable to finite samples}
\label{fig:mc}
\end{figure}

\subsection{Voting and the maximum over coordinates}
Use $n=10$, $\gamma=0.25$, $c_0=\mathbf1$, $r_0=1$, $x^*=0$ and $K=40$.
For each odd $m\in\{1,3,5,9,15,25,41\}$, perform $800$ independent runs.
At each coordinate query the majority is incorrect with its exact binomial
tail probability $p_m$; drawing this Bernoulli event is distributionally
identical to drawing the $m$ independent votes explicitly. The charged
information cost is still $m$ noisy vector queries per step.

Figure~\ref{fig:vote} contrasts coordinatewise and maximum-norm guarantees.
\begin{figure}[htbp]
\centering\includegraphics[width=\textwidth]{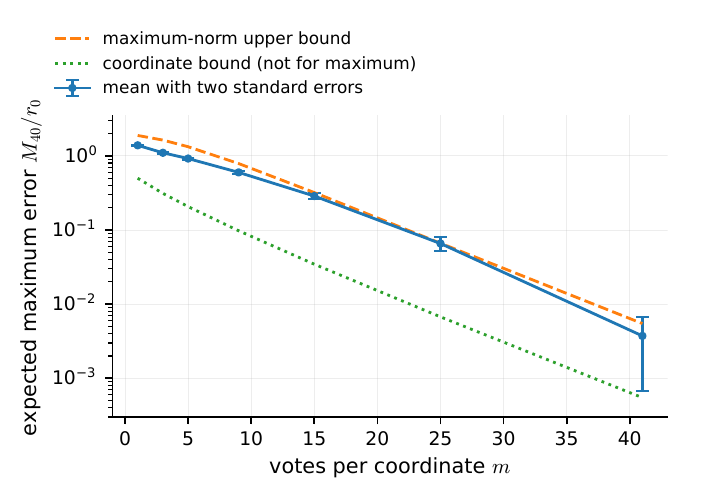}
\caption{Noisy voting: empirical mean of $M_{40}$ with two standard errors,
compared with $2r_0[1-(1-p_m)^n]+R_{40}$. The per-coordinate bound
$2p_mr_0+R_{40}$ is also displayed and cannot be substituted for a bound on
the maximum norm. Rare failures make small-sample estimates noisy at large
$m$; the proof, not empirical coverage of the bars, supplies the upper bound}
\label{fig:vote}
\end{figure}

\subsection{Freezing and independent restarts}
On the quadratic of Counterexample~\ref{cex:2x2}, a threshold just above
$1/11$ skips the initially wrong, small first derivative. Figure~\ref{fig:freeze}
uses the same $120$ full-gradient queries for all thresholds. It illustrates
sensitivity to the active set, not membership in $\DSC_\epsilon$ or uniform
improvement over all targets.

\begin{figure}[htbp]
\centering\includegraphics[width=\textwidth]{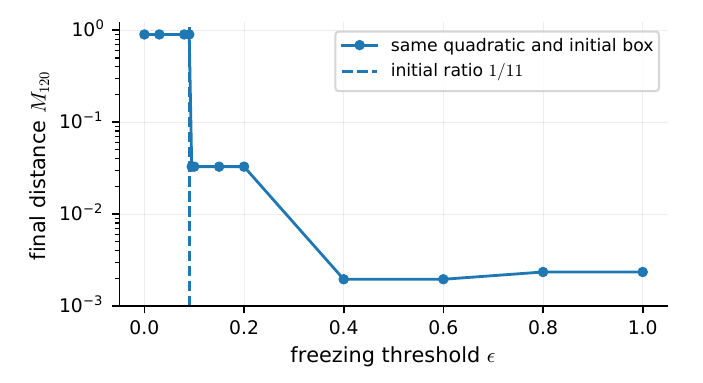}
\caption{Freezing on one two-dimensional instance. Both the centre and
half-size of an inactive coordinate remain unchanged. The vertical line is
the initial derivative ratio $|\partial_1f(c_0)|/\|\nabla f(c_0)\|_\infty=1/11$.
Final errors and coordinate update counts are saved; small initial derivatives
are not generally synonymous with incorrect signs}
\label{fig:freeze}
\end{figure}

For multistart, return to the quadratic--cosine family at $A=1$ in dimension
$3$. Each run starts uniformly on $[-2,2]^3$, uses $r_0=0.5$ and $40$
halving steps. Declare success by final distance at most $0.05$ from the
known global minimiser $0$. For each $M\in\{1,2,5,10,20,40,80\}$, repeat
$M$ independent runs in each of $200$ independent outer trials. Record
separately whether some candidate succeeds and whether the candidate with
the smallest final objective succeeds. Estimate the single-start basin mass
from a separate sample of $4000$ runs. This is a basin experiment: not every
initial box contains the global minimiser.

Figure~\ref{fig:multistart} reports the two notions of multistart success.
\begin{figure}[htbp]
\centering\includegraphics[width=\textwidth]{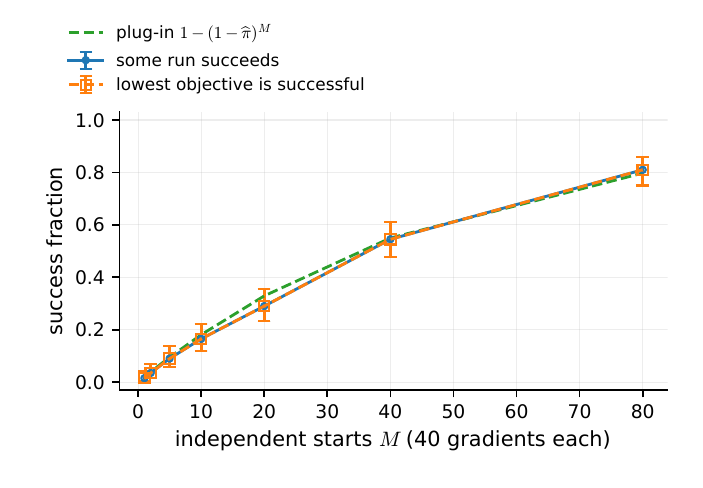}
\caption{Multistart success from independent random centres. Bars are nominal
$95\%$ Wilson intervals. The dashed curve uses an independently estimated
basin mass in $1-(1-\widehat\pi)^M$; it is a plug-in prediction, not an exact
finite-sample guarantee. Existence and objective-based selection are recorded
as separate events; they coincide in every recorded trial here}
\label{fig:multistart}
\end{figure}

\FloatBarrier

\section{Reproduction protocol}\label{app:repro}

\subsection{Files and commands}
Online Resource~1 provides the Python sources, raw CSV tables, saved quadratic
instances, numerical verification results, and the sources for every figure.
The numerical figures are produced by Python, and the two explanatory diagrams
by TikZ. From the resource root directory, run
\begin{verbatim}
python code/verify_results.py
python code/publication_figures.py
python code/additional_figures.py
python code/publication_layout.py
python illustrations/build_diagrams.py
\end{verbatim}
The scripts resolve paths relative to their own locations, not the invoking
working directory. They write data to \texttt{data} and vector PDFs to
\texttt{figures}. They use NumPy, SciPy and Matplotlib; the exact checks use
Python's \texttt{fractions.Fraction}. The recorded execution environment is
Python 3.13.5, NumPy 2.3.5, SciPy 1.17.0 and Matplotlib 3.10.8, with one BLAS
thread. The supplied requirements file specifies these versions. Numerical
round-off may differ across BLAS implementations; data values, rather than
PDF metadata timestamps, are the reproducibility target.

The publication master seed is $20260919$. Each random stream uses a seed sequence with entries
\texttt{master}, \texttt{series} and \texttt{case}. Trials within a vectorised
sample use independent draws from that addressed stream; they are not claimed
to have separate generators when they do not. The additional exploratory
scripts \texttt{exp\_a} through \texttt{exp\_g}, provided as supplementary
computations, use their own documented seeds and budgets. Their sequential
random-number streams should not be confused with the case-addressed
publication protocol. All ten supplementary entry-point scripts were executed successfully.

\subsection{Exact and numerical checks}
The file \texttt{verification.json} lists case counts, tolerances, maximum
observed numerical discrepancies and a pass/fail summary. Exact equality is
required in rational tests. Retention and rate tests allow only documented
floating-point tolerances, and save the largest violation relative to the
proposed upper bound rather than silently clipping it to zero. The envelope
witness is evaluated through its proved analytic sign field, not through
Monte Carlo approximation of $G$ in~\eqref{eq:G-potential}.

For $\beta=0.9999$, $\theta=0.99999$, $r_0=1$, the scalar envelope has
$b_{400}\approx1.03514$, yet its eventual limit is about $1.47112$.
Consequently a fixed-budget envelope evaluator must not return its current
value as an upper bound on the limit before the switch is certified. The
implementation returns the safe bound $b_K+R_K$ when no switch has yet been
detected: all future increases sum to at most $R_K$. Once the first branch
is permanent, it returns $b_J-R_J$. This behaviour has a dedicated regression
test.

For Perron weights, the symmetric nonnegative matrix
$S=D^{-1/2}|H-D|D^{-1/2}$ is split into connected components and its largest
\emph{algebraic} eigenvalue is used on each nontrivial block. Isolated
coordinates receive positive weights. This avoids selecting a negative
extremal eigenvector in a bipartite block or zeroing entire blocks of a
reducible matrix. After numerical normalisation, the implementation recomputes
$\max_i(Aw)_i/w_i$ from the returned weights. A floating-point eigenvector
is a diagnostic approximation, not an interval-arithmetic certificate.

\subsection{Matched quadratic instances and costs}
The diagonal benchmark uses $H=\diag(\operatorname{geomspace}(1,1000,50))$.
For the dense benchmark, generate $Z_{ij}\sim N(0,1/n)$ independently, set
$H=I+0.35(Z+Z^\top)/\sqrt2$, and add
$[0.05-\lambda_{\min}(H)]_+I$. Save $H$, $x^*$, $c_0$ and $r_0$ in an
NPZ file. The target is drawn after the matrix using the same case-addressed
stream, so both the formula and saved instance define the experiment.

The iRprop$^-$ implementation suppresses a coordinate update on a derivative
sign reversal, then stores a zero previous derivative in that coordinate.
The distinction matters: reducing the step while still making that update is
a different algorithm. Gradient descent, Adam and signGD are implemented in
the same module, with parameters reported in the main text. Neither their
step tuning nor evaluation of $L$ is free auxiliary information in an oracle
complexity comparison. The benchmark is a controlled numerical comparison,
not an information-theoretic equivalence between these methods.

All counts distinguish gradient vectors, binary comparisons and scalar values.
A quadratic comparison identity can be evaluated without subtraction noise
when $H$ is explicitly known, but black-box evaluations of two nearby values
must account for cancellation and external noise through
Corollary~\ref{cor:fd}. Adaptive geometry also pays for $Ar_k$; Monte Carlo
smoothing for $m$ gradients per step; majority voting for $m$ noisy sign
vectors; multistart for all runs and final values used in selection. These
costs are reported separately rather than absorbed into a nominal iteration.

\subsection{Scope of the computational evidence}
The numerical plots are regenerated by the supplied scripts; the two
analytical diagrams are constructed from the TikZ sources supplied in Online Resource~1. Data and
figures establish reproducibility of the stated finite tests, not formal
verification of all proofs, exhaustive optimisation of reference-method
parameters, or performance on an unspecified real-world distribution.
The exact sharpness constructions, lower bounds and retention theorems are
proved analytically; stochastic plots indicate their practical limitations
under explicitly stated sampling models.

\end{appendix}
\end{document}